\documentclass[11pt,a4paper]{article}
\usepackage{times}
\usepackage{tgtermes}
\usepackage{amsmath, amsthm, amscd, amsfonts, amssymb, graphicx, color}\usepackage[T1]{fontenc}
\usepackage{fancyhdr}
\usepackage{setspace}
\begin{document}

\newtheorem{proposition}{Proposition}[section]
\newtheorem{definition}[proposition]{Definition}
\newtheorem{theorem}[proposition]{Theorem}
\newtheorem{corollaire}[proposition]{Corollary}
\newtheorem{lemma}[proposition]{Lemma}
\newtheorem{remark}[proposition]{Remark}

\begin{center}
\begin{spacing}{1.2}
{\fontsize{16}{20}
\textsc{Eigenvalues of Dirichlet Laplacian within the class of open sets with constant diameter}
}
\end{spacing}
\end{center}
\vspace{-0.738cm}
\begin{center}
{\fontsize{13}{20}
\textsc{Z. Fattah}, \textsc{M. Berrada} \\
\bigskip
}
{\fontsize{11}{20}
Mathematics and computer science department, ENSAM of Mekn\`es, University of Moulay Ismail, Morocco\\ z.fattah@edu.umi.ac.ma, m.berrada@ensam.umi.ac.ma}
\end{center}

\vspace{10pt}

{\fontsize{12}{20}
\textsc{Abstract}:
}
This paper is about a shape optimization problem related to the Dirichlet Laplacian eingevalues in the Euclidean plane. More precisely we study   the shape of the minimizer in the class of open sets of constant width. We prove that
 the disk is not a local minimizer except for a limited number of eigenvalues.

\vspace{20pt}

{\fontsize{12}{20}
Keywards: eigenvalue, convex bodies, Dirichlet, Laplacian, shape optimization, constant width, Bessel functions
}

%
\newtheorem{assumption}{Assumption}%
%


\section{Introduction}
\label{intro}
\medskip
The shape optimization of the eigenvalues of an elliptic operator is an old problem.
Lord Rayleigh  considered the Laplacian operator with Dirichlet boundary conditions in his book "Theory of sound" \cite{ray1896}.
He stated that
among sets of fixed measure,
the disk minimizes the first eigenvalue in the Euclidean plane.  The proof of this statement came later simultaneously and independently by G.Faber 1923 \cite{faber1923} and E.Krahn 1924 \cite{krahn1926}.
A natural question is the optimal shape for the other eigenvalues. The second eigenvalue was studied by E. Krahn \cite{krahn1926}, Szeg$\ddot{o}$ \cite{polya1955} and I. Hong \cite{hong1954} who proved that the minimum among sets of constant measure is the union of two identical balls.
 Next, S.A. Wolf and J.R. Keller \cite{wolf1994} proved that, for the third eigenvalue, the disk is a local minimum among sets of constant measure in the plane. The global minimizer remains nowadays an open problem even it is conjectured that the disk is the global minimizer and all numerical computations show that.
 Recently, A.Berger \cite{berger2015} proved that except the first and the third ones, no eigenvalue can be minimized by the disk.

Our paper is focused on the minimization of the eigenvalues of the Dirichlet Laplacian in $\mathbb{R}^{2}$ with a different constraint: we assume that our sets have constant diameter. Indeed we study the following shape optimization problem:
\begin{equation}\label{intro2}
  min\{\lambda_{\kappa}(\Omega),\; \Omega\subset \mathbb{R}^N \;\mbox{open set such that}\; D(\Omega)=\alpha\},
\end{equation}
where $(\lambda_{\kappa}(\Omega))_{k\in \mathbb{N}^{*}}$ denotes the eigenvalue of the Dirichlet-Laplacian, $D(\Omega)$ denotes the diameter of $\Omega$ and $\alpha\in (0,+\infty)$. We explicitly observe that our problem is equivalent to
\begin{equation}\label{intro1}
  min\{\lambda_{\kappa}(\Omega),\; \Omega\subset \mathbb{R}^N\;\mbox{open set of constant width}\;\alpha\},
\end{equation}
since an arbitrary set is contained in a set of constant width of the same diameter.

After proving the existence of a solution to problem (\ref{intro2}) for every $\kappa\in \mathbb{N}^{*}$,  we show that the ball is the solution for  problem (\ref{intro2}) when $\kappa=1$.
Then we study the local minimality of the disk for problem (\ref{intro2}),
in the spirit of \cite{berger2015} and \cite{wolf1994}. Our results are:
      \begin{enumerate}
        \item for $\kappa\in \{1, 3, 5, 8, 12, 17, 27, 34, 42\}$, the disk is a local minimizer (for smooth deformations) for problem (\ref{intro2}).
        \item for $\kappa\in \mathbb{N}^{*}\setminus \{1, 2, 3, 4, 5, 7, 8, 11, 12, 16, 17, 26, 27, 33, 34, 41, 49, 50\}$ the disk is not a local minimizer for problem (\ref{intro2}).
      \end{enumerate}

We were not able to answer to the question of the local minimality of the disk for the  cases $\kappa\in \{2,4,7,11,16,26,33,41,49;50\}$.

The paper is organized as follows. In section \ref{sec:main}, we recall some definitions and properties for Dirichlet-Laplacian eigenvalue and their continuity with respect to the $\gamma-$ convergence and the Hausdorff convergence. After that, we introduce also the notion of convex body with constant width. In section \ref{firstmainresult} we prove the existence of a solution to problem \ref{intro2} and we study the optimal domain for $\lambda_{1}$ and $\lambda_{3}$. In section \ref{sec:bodies}, we define a smooth deformation of the disk among open sets of constant width and we write the polar parametrization of this family from Gauss parametrization. Section \ref{sec:Dirichlet} is devoted to the computation of the asymptotic expansion of the eigenvalues, with respect to our  deformation of the disk. We distinguish two cases: simple eigenvalue and double eigenvalue. Finally, in section \ref{sec:Results}, we prove the main results of this paper by giving the eigenvalues locally minimized by the disk (Theorem \ref{Eigenvalues}) and the eigenvalue which are not minimized by the disk (Theorem \ref{X}).
\section{Preliminaries}
\label{sec:main}

\newcommand{\jnp}{j_{m,p}}
\newcommand{\jnpp}{j_{m,p}^{2}}
Let $\Omega$ be a bounded open set of $\mathbb{R}^{N}$ and let us denote by $0<\lambda_{1}(\Omega)\le \lambda_{2}(\Omega)\le \lambda_{3}(\Omega)\le \cdots$ the eigenvalues of the Laplacian with Dirichlet boundary conditions. The corresponding eigenfunctions $u_1,\; u_2,\; u_3,\dots$ satisfy (in a variational sense)

\begin{equation}\label{eq.0}
\left\{
  \begin{array}{ll}
    -\Delta u_{\kappa}=\lambda_{\kappa}u_{\kappa}, & \mbox{ in } \Omega \\
    u_{\kappa}=0, & \mbox{ on } \partial \Omega.
  \end{array}
\right.
\end{equation}

We recall that, by the classical $\min-\max$ formula of Courant and Fisher for eigenvalues,  the following monotonicity for the inclusion holds:
\begin{equation}\label{monot}
\Omega_1\subset \Omega_2 \Rightarrow \lambda_{\kappa}(\Omega_1)\geq \lambda_{\kappa}(\Omega_2).
\end{equation}

We  recall  two famous theorems which we are going to use in the sequel. For the proofs, see respectively \cite{henrot2006} and \cite{wolf1994}.

\begin{theorem}[Faber-Krahn]\label{Faber-Krahn}
\begin{equation}\label{Eq.FK}
\lambda_{1}(B)=\min\{\lambda_{1}(\Omega),\; \Omega\subset\mathbb{R}^{N} \mbox{ open, } |\Omega|=1\}
\end{equation}
where $B$ is the ball of volume 1.
\end{theorem}

\begin{theorem}[Wolf-Keller]\label{Wolf-Keller}
$\lambda_{3}$ is locally minimized by the disk among the sets of constant measure.
\end{theorem}

To prove existence of minimizers for eigenvalues, we obviously need continuity of eigenvalues with respect to the domain. Let us recall some definitions and theorems used in the sequel.

\begin{definition}[Hausdorff distance]

Let $K_{1}$ and $K_{2}$ be two non-empty compact sets in $\mathbb{R}^{N}$. We set
$$
  \displaystyle\forall x\in \mathbb{R}^{N}, d(x,K_{1}):=\inf_{y\in K_{1}}|y-x|
$$
$$
  \displaystyle \rho(K_{1}, K_{2}):=\sup_{x\in K_{1}}d(x,K_{2})
$$
Then the Hausdorff distance of $K_{1}$ and $K_{2}$ is defined by
\begin{equation}\label{H1}
  \displaystyle  d^{H}(K_{1},K_{2}):=\max(\rho(K_{1}, K_{2}),\rho(K_{2}, K_{1}))
\end{equation}
\end{definition}

For open sets, we define the Hausdorff distance through their complementary:

\begin{definition}
Let $\Omega_{1}$, $\Omega_{2}$ be two open subsets of a (large) compact set $B$. Then their Hausdorff distance is defined by:

\begin{equation}\label{H2}
  \displaystyle d_{H}(\Omega_{1},\Omega_{2}):= d^{H}(B\setminus\Omega_{1},B\setminus\Omega_{2})
\end{equation}
\end{definition}

\begin{definition}[$\gamma-$convergence]

Let $B$ be a ball ,\; $\Omega_n \subset B $ be a sequence of open sets and $\Omega \subset B$ be an open set. We say that $\Omega_n$ $\gamma-converges $ to $\Omega$ for every $f\in L^2(B)$ the solution $u^{f}_{\Omega_n}$ of the Dirichlet problem for the Laplacian on $\Omega_n$ with right-hand side $f$ converges (strongly) in $L^2(B)$ to $u^{f}_{\Omega}$.
\end{definition}

\begin{theorem}\label{continuity}
  Let $B$ be a fixed compact set in $\mathbb{R}^{N}$ and $\Omega_{n}$ be a sequence of convex open sets in $B$ which converges, for the Hausdorff metric, to a (convex) set $\Omega$. Then $\Omega_{n}$ $\gamma-$ converges to $\Omega$ and, in particular, for all $\kappa$ fixed, $\lambda_{\kappa}(\Omega_{n})\rightarrow \lambda_{\kappa}(\Omega)$.
\end{theorem}
\begin{proof}
  \cite{henrot2006}, p:31.
\end{proof}

\begin{theorem}\label{Property}
  Let $\Omega_{n}$ and $\Omega$ be bounded open subsets of $\mathbb{R}^{N}$ such that $\Omega_{n}$ converges to $\Omega$ in sense of Hausdorff metric. If $K\subset \Omega$ is compact there exists $n_{K}\in \mathbb{N}$ such that $K\subset \Omega_{n}$ for all $n\geq n_{K}$.
\end{theorem}

\begin{proof}
  \cite{henrot2005}, p:32.
\end{proof}

We are now going to define the width of a convex set.

\begin{definition}
 Let $C$  be a non empty closed and bounded convex set in $\mathbb{R}^{2}$. The support function $\sigma_{C}$  of $C$ is defined by:
$$
\begin{array}{rcl}
\sigma_{C}:\mathbb{R}^{2} & \rightarrow & \mathbb{R} \\
 n &\mapsto                             &  \sigma_{C}(n)=\displaystyle \max_{c\in C}<c,n>
\end{array}
$$
where $<,>$ denotes the scalar product.
\end{definition}
The support function can be equivalently defined on the unit sphere $\mathbb{S}^{1}$ by homogeneity:
$$
\displaystyle h_{C}:\varphi\in \mathbb{R}\mapsto h_{C}(\varphi)=\sigma_{C}(\cos (\varphi),\sin (\varphi))\,.
$$

The support function $\sigma_{C}(n)$ is the distance of the support line  $D_{\varphi}$ given by the equation $\cos (\varphi)x+\sin (\varphi)y=h_{C}(\varphi)$ from the origin, where $(x,y)$ are the coordinates of a point in the Euclidian plane. For more details about the support functions see \cite{rolf2013} and \cite{convex}.\\

The support function of a convex $C\subset\mathbb{R}^{2}$ is of class $C^{1}$ in $\mathbb{R}^{2}\setminus\{0\}$ if and only if $C$ is strictly convex. In this case, the boundary $\partial C$ can be described as follows:
\begin{equation}\label{first}
  \left\{
    \begin{array}{ll}
      x(\varphi)=h(\varphi)\cos(\varphi)-h^{'}(\varphi)\sin(\varphi) \\
      y(\varphi)=h(\varphi)\sin(\varphi)+h^{'}(\varphi)\cos(\varphi)
    \end{array}
  \right.
\end{equation}
where the prime denoted the differentiation.

If $h_{C}$ is of class $C^{1,1}$, $h_{C}^{''}$ exists almost everywhere by Rademacher's theorem. The quantity $\rho=h_C+h_C^{''}$ is the positive radius of curvature of the boundary of $C$.

A convex body is a nonempty compact convex subset of $\mathbb{R}^N$. The following lemma gives an important property of convex bodies.

\begin{lemma}\label{Lemme T.bayen}
Let $h$ be a function twice differentiable in $]-\infty,+\infty[$ of period $2\pi$. $h$ is a support function of a convex body on $\mathbb{R}^{2}$ if for all $\varphi \in [0,2\pi]$
$$
h(\varphi)+h^{''}(\varphi)>0.
$$
\end{lemma}

\begin{proof}
  See \cite{gro}, \cite{sheph} and \cite{rolf2013}.
\end{proof}
\begin{definition}
Let $C$ be a convex body of support function $\sigma_{C}$. The width of $C$ in the direction $u\in \mathbb{S}^{1}$ is $\sigma_{C}(u)+\sigma_{C}(-u)$.
\end{definition}

A convex body $C$ is of constant width if $\sigma_{C}(u)+\sigma_{C}(-u)=\alpha$ for all $u\in \mathbb{S}^{1}$ ($\alpha=2$, in our case). The quantities
$\sigma_{C}(u)+\sigma_{C}(-u)$ represent the distance between two different parallel support lines to $C$. One can see that
$h_{C}(\varphi)+h_{C}(\varphi+\pi)=2$ $\forall \varphi\in [0,2\pi]$.

\section{Study of $\lambda_{1}$ and $\lambda_{3}$}\label{firstmainresult}

The next theorem shows the existence of a solution for problem (\ref{intro2}). Obviously, it is equivalent to consider the constraint $D(\Omega)\leq \alpha$ or $D(\Omega)=\alpha$.

\begin{theorem}
For every  $k\in \mathbb{N}^{*}$ the problem
\begin{equation}
   \min\{\lambda_{\kappa}(\Omega),\; \Omega \mbox{open}\;\subset \mathbb{R}^{N},\; D(\Omega)\leq \alpha\}
\label{lamk}
\end{equation}
has at least a convex solution.
\end{theorem}

\begin{proof}
Let $\Omega_{n}$ be a minimizing sequence. It is clear that $\Omega_{n}\subset conv(\Omega_{n})$ where $conv(\Omega_{n})$ is the convex hull of $\Omega_{n}$ and  that $conv(\Omega_{n})$ have the same diameter as $\Omega_{n}$ (see \cite{edgar}, p.166). By (\ref{monot}),
$\lambda_{\kappa}(conv(\Omega_{n}))\leq \lambda_{\kappa}(\Omega_{n})$.
Therefore $(conv(\Omega_{n}))_{n}$ is also a minimizing sequence.

Since the diameter of  $conv(\Omega_{n})$ is smaller than $\alpha$,  $conv(\Omega_{n})$ is a bounded sequence. Thus we can extract a sub-sequence still denoted by $conv(\Omega_{n})$ such that $conv(\Omega_{n})$ converges to $\Omega$ for the Hausdorff metric. Since $conv(\Omega_{n})$ and $\Omega$ are convex. By theorem \ref{continuity},  $conv(\Omega_{n})$ converges to $\Omega$ in the $\gamma$-convergence sense . In particular, for all fixed $\kappa$,
$\lambda_{\kappa}(conv(\Omega_{n}))$ converges to $\lambda_{\kappa}(\Omega)$ by the  $\gamma$-continuity of eigenvalues (see theorem \ref{continuity}).

We notice that the limit $\Omega$ is a "true" domain (i.e., it is not the empty set). Indeed if a minimizing sequence converge to the empty set, then the eigenvalue $\lambda_{\kappa}$ goes to infinity, which is absurd.

It remains to prove that
$D(\Omega)\leq \alpha$. By contradiction, we assume that $D(\Omega)> \alpha$. For $\varepsilon>0$, we can find a compact set $K_{\varepsilon}\subset \Omega$ such that $D(K_{\varepsilon})=D(\Omega)-\varepsilon >\alpha$. By theorem \ref{Property}, there exists  $\displaystyle n_{K_{\varepsilon}}\in \mathbb{N}$ such that $K_{\varepsilon}\subset conv(\Omega_{n})$ for all $\displaystyle n\geq n_{K_{\varepsilon}}$. This implies that $D(conv(\Omega_{n}))> \alpha$ which is absurd.
\end{proof}

We recall the isodiametric inequality that we use to prove our first theorem.

\begin{theorem}[Isodiametric inequality]
  For all set $A\subset \mathbb{R}^{N}$
  $$
  |A|\leq |B_{1}|\bigg(\frac{diam(A)}{2}\bigg)^{N}
  $$
where $B_{1}$ denotes the unit ball in the Euclidean space $\mathbb{R}^{N}$ and $|.|$ is the Lebesque measure.
\end{theorem}

\begin{proof}
  \cite{evans}, p:69.
\end{proof}

\begin{theorem}\label{iso}
\begin{equation}\label{Eq.R1}
  \lambda_{1}(B)=\min\{\lambda_{1}(\Omega),\; \Omega \;\mbox{open}\; \subset \mathbb{R}^{N}\; D(\Omega)=2\}
\end{equation}
where $B$ is the ball of diameter equals to $2$.
\end{theorem}
\begin{proof}
Let $\Omega$ be an open set of $\mathbb{R}^{N}$ with $D(\Omega)=2$. Let $\Omega^{*}$ be the ball of the same volume as $\Omega$. According to  Faber-Krahn's theorem (Theorem \ref{Faber-Krahn}), we have
$
  \lambda_{1}(\Omega^{*})\leq \lambda_{1}(\Omega).
$

Let $B$ be the ball of diameter equals to $2$. By the isodiametric inequality we have $
  |\Omega^{*}|=|\Omega|\leq |B|.
$
Since $\Omega^{*}$ is a ball, then $\Omega^{*} \subset B$. By (\ref{monot})
$
 \lambda_{1}(B)\leq \lambda_{1}(\Omega^{*}) \mbox{ so that } \lambda_{1}(B)\leq  \lambda_{1}(\Omega).
$
\end{proof}
\begin{theorem}
  $\lambda_{3}$ is locally minimized by the disk among sets of constant width.
\end{theorem}

\begin{proof}

Let $\varepsilon>0$. Let $B$ be the open disk of diameter $\alpha=2$. Let $C$ be an open convex of constant width $\alpha$ such that
$d_{H}(C,B)< \varepsilon$. By  the isodiametric inequality, we have  $|C|\leq |B|$. By theorem \ref{Wolf-Keller}
 $$
 \lambda_3(B)\leq \lambda_3(C).
 $$

\end{proof}

\section{Bodies of constant width obtained by a small deformation of a disk}\label{sec:bodies}

We are going to study the minimality of Dirichlet-Laplacian in a neighborhood of the disk among open sets of constant width. The question here is how to construct this neighborhood, we shall consider convex bodies of constant width near to the disk. The most confident way is to perturb the support function.\\

The support function of the unit disk $D$ is given by $h_D(\varphi)=1\; \forall \varphi\in [0,2\pi]$. We are going to study some sufficient conditions which guarantee that $1+\varepsilon f(\varphi)+\varepsilon^{2} g(\varphi)$ is a support function, for some functions  $f$, $g$ on $[0,2\pi[$.

\begin{lemma}\label{zak1}
Let $\alpha$, $\varepsilon$ be  positif real numbers and $h$ a function defined by
\begin{equation}
h(\varphi)=1+\varepsilon f(\varphi)+\varepsilon^{2} g(\varphi)\quad \forall \varphi\in \mathbb{R}
\label{hep}
\end{equation}
where $f(\varphi)=\sum_{-\infty}^{+\infty}a_{n}e^{in\varphi}$, $g(\varphi)=\sum_{-\infty}^{+\infty}b_{n}e^{in\varphi}$,  $a_{-n}=\overline{a_{n}}$ and $b_{-n}=\overline{b_{n}}$. If $a_{n}\in O(\frac{1}{n^{3+\alpha}})$, $b_{n}\in O(\frac{1}{n^{3+\alpha}})$ and $a_{2n}=b_{2n}=0$ for all n integers, then there exists a convex body $\Omega_{\varepsilon}$ of constant width equal $2$  such that $h=h_{\Omega_{\varepsilon}}$.
\end{lemma}

\begin{proof}
It is clear that $h$ is twice differentiable by the assumption $a_{n}, b_{n}\in O(\frac{1}{n^{3+\alpha}})$. It is $2\pi-$ periodic by construction and $h+h^{''}>0$ for small $\varepsilon$. By lemma \ref{Lemme T.bayen}, $h$ is a support function of a convex body $\Omega_{\varepsilon}$.

Since $a_{2n}=b_{2n}=0$, one has $h(\varphi)+h(\varphi+\pi)=2$ for every $\varphi\in \mathbb{R}$, that is,  $h$ is a support function of a convex body with constant width $\alpha=2$.
\end{proof}

\begin{assumption}\label{asumption}
We shall, from now on, assume that the sequences $a_{n}$ and $b_{n}$ satisfy the assumptions of  lemma \ref{zak1}.
\end{assumption}

In order to calculate the eigenvalues in $\Omega_{\varepsilon}$, we need to calculate the radius of $\Omega_{\varepsilon}$ as the authors did in \cite{ray1896}, \cite{wolf1994} and \cite{berger2015} for open sets of fixed measure. We have two possible parameterizations of $\partial \Omega_{\varepsilon}$:
the parametrization (\ref{first}) with the support function and the {polar} parametrization $(R(\theta,\varepsilon),\theta)$.

We are now going to write the last parametrization. Since $\Omega_{\varepsilon}$ is a convex, it is a star shaped set. Furthermore, according to (\cite{morvan}, p.80) $d_{H}(\Omega_{\varepsilon},D)=||h_{\Omega_{\varepsilon}}-h_{D}||_{\infty}=\displaystyle\sup_{\varphi\in [0,2\pi[}||\varepsilon f(\varphi)+\varepsilon^{2}g(\varphi)||$.

This implies that the radius $R(\theta,\varepsilon)$ can be written as
\begin{equation}
R(\theta,\varepsilon)=1+\varepsilon F(\theta)+\varepsilon^{2}G(\theta)+O(\varepsilon^{3}).
\label{eq.22}
\end{equation}

\begin{figure}[!h]
  \centering
  \includegraphics[width=6.5cm]{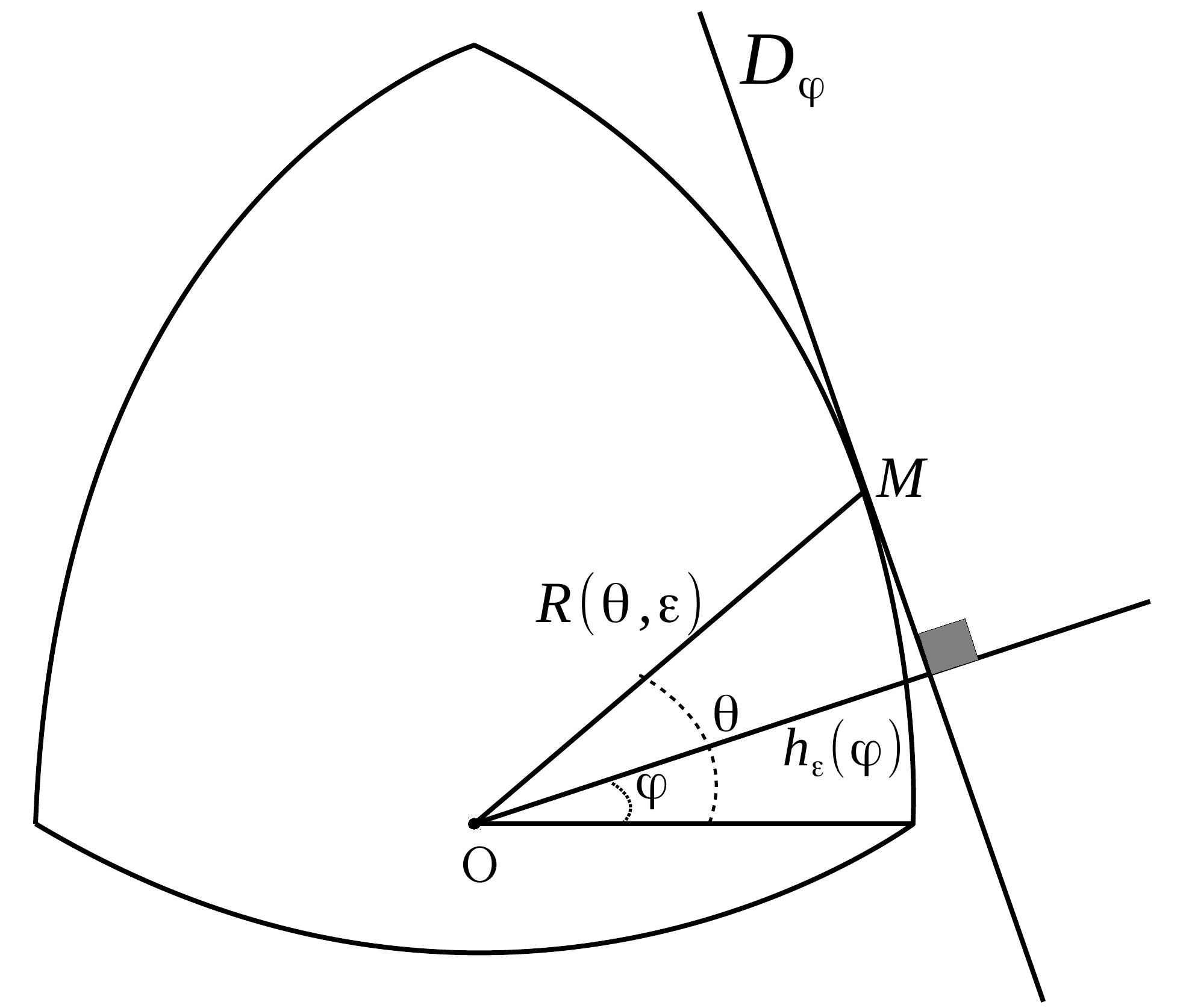}
  \caption{An example of convex body with constant width.}\label{Fig1}
\end{figure}

As stated in lemma \ref{Lemme T.bayen}, the support function $h_{\Omega_{\varepsilon}}$ is $C^{2}$. Therefore, the boundary of $\Omega_{\varepsilon}$ is a regular sub-manifold of class $C^{2}$ (\cite{rolf2013}, p:111). Thus the radius $R(\theta,\varepsilon)$ of the polar parametrization of $\Omega_{\varepsilon}$ is also $C^{2}$.
\begin{definition}[Big $O$ notation-Landau symbol]
Let $f$ and $g$ be a positive real functions and $a\in \mathbb{R}$. We write $f=O(g)$ as $x\rightarrow a$ if there are two positive numbers $\delta$ and $A$ so that $|f(x)|< A |g(x)|$ whenever $|x-a|< \delta$.
\end{definition}
\begin{lemma}
Let $\Omega_{\varepsilon}$ be a convex body of constant width obtained by the above deformation of the disk. The radius in a given point $M\in \partial \Omega_{\varepsilon}$ is
\begin{equation}
R(\theta,\varepsilon)=1+\varepsilon f(\theta)+\varepsilon^{2}\Bigg( g(\theta)-\frac{1}{2}(f^{'}(\theta))^{2}\Bigg)+O(\varepsilon^{3}).
\label{eqR2}
\end{equation}
\end{lemma}

\begin{proof}
Let $h_{\Omega_{\varepsilon}}(\varphi)=1+\varepsilon f(\varphi)+\varepsilon^{2} g(\varphi)$ be the support function of $\Omega_{\varepsilon}$ satisfying the assumptions of lemma \ref{zak1}.

The polar parametrization of $\Omega_{\varepsilon}$ is given by
\begin{equation}\label{eq.44}
  \left\{
    \begin{array}{ll}
     x(\theta)=R(\theta,\varepsilon)\cos(\theta) \\
     y(\theta)=R(\theta,\varepsilon)\sin(\theta)
    \end{array}.
  \right.
\end{equation}

For a given $\varepsilon >0$, we denote $R(\theta)= R(\theta, \varepsilon)$. The exterior normal vector to the boundary of $\Omega_{\varepsilon}$ is defined by
\begin{eqnarray}\label{normal1}
  n      &=& \frac{1}{\sqrt{R^{2}(\theta)+R^{'2}(\theta)}}\Bigg[R(\theta)\left(\begin{array}{c}
                                                 \cos(\theta) \\
                                                 \sin(\theta)
                                               \end{array}\right)-R^{'}(\theta)\left(\begin{array}{c}
                                                 -\sin(\theta) \\
                                                 \cos(\theta)
                                               \end{array}\right)\Bigg]\,.
\end{eqnarray}
The parametrization of $\Omega_{\varepsilon}$ using the support function is
\begin{equation}\label{eq.55}
  \left\{
    \begin{array}{ll}
      x(\varphi)=h_{\Omega_{\varepsilon}}(\varphi)\cos(\varphi)-h_{\Omega_{\varepsilon}}^{'}(\varphi)\sin(\varphi) \\
      y(\varphi)=h_{\Omega_{\varepsilon}}(\varphi)\sin(\varphi)+h_{\Omega_{\varepsilon}}^{'}(\varphi)\cos(\varphi)\,.
    \end{array}
  \right.
\end{equation}
The exterior normal is
\begin{equation}\label{normal2}
n=\left(
\begin{array}{c}
   \cos(\varphi) \\
  \sin(\varphi)
\end{array}\right).
\end{equation}

Equations (\ref{normal1}) and (\ref{normal2}) give
$$
\left(
\begin{array}{c}
   \cos(\varphi) \\
  \sin(\varphi)
\end{array}\right)=\left(\begin{array}{cc}
                    \cos(\theta)  & -\sin(\theta) \\
                     \sin(\theta) & \cos(\theta)
                   \end{array}\right)\left(
\begin{array}{c}
  \frac{R(\theta)}{\sqrt{R^{2}(\theta)+R^{'2}(\theta)}} \\
  \frac{-R^{'}(\theta)}{\sqrt{R^{2}(\theta)+R^{'2}(\theta)}}
\end{array}\right),
$$
so,
$$
arg\left(
\begin{array}{c}
   \cos(\varphi) \\
  \sin(\varphi)
\end{array}\right)=\theta+arg\left(
\begin{array}{c}
  \frac{R(\theta)}{\sqrt{R^{2}(\theta)+R^{'2}(\theta)}} \\
  \frac{-R^{'}(\theta)}{\sqrt{R^{2}(\theta)+R^{'2}(\theta)}}
\end{array}\right)
$$

where arg denoted the argument. Therefore,
\begin{eqnarray*}
 arg\left(
\begin{array}{c}
  \frac{R(\theta)}{\sqrt{R^{2}(\theta)+R^{'2}(\theta)}} \\
  \frac{-R^{'}(\theta)}{\sqrt{R^{2}(\theta)+R^{'2}(\theta)}}
\end{array}\right)  &=& \displaystyle \arctan \Bigg( \frac{\frac{-R^{'}(\theta)}{\sqrt{R^{2}(\theta)+R^{'2}(\theta)}}}{\frac{R(\theta)}{\sqrt{R^{2}(\theta)+R^{'2}(\theta)}}}\Bigg)\\
                    &=& \arctan\Bigg(\frac{-R^{'}(\theta)}{R(\theta)}\Bigg) \\
                    &=& \frac{-R^{'}(\theta)}{R(\theta)}+ O(\varepsilon^{2})
\end{eqnarray*}

That implies

$$
  \varphi= \theta-\frac{R^{'}(\theta)}{R(\theta)}+O(\varepsilon^{2}).
$$
%

Let $M$ be the point defined by $\partial\Omega_{\varepsilon}\cap D_{\varphi}=\{M\}$ (see figure \ref{Fig1}) where $D_{\varphi}$ is the support line defined with normal vector

\begin{eqnarray*}
   n   &=& (\cos(\varphi),\sin(\varphi))\\
       &=& \frac{1}{\sqrt{R^{2}(\theta)+R^{'2}(\theta)}}\Bigg[R(\theta)\left(\begin{array}{c}
                                                 \cos(\theta) \\
                                                 \sin(\theta)
                                               \end{array}\right)-R^{'}(\theta)\left(\begin{array}{c}
                                                 -\sin(\theta) \\
                                                 \cos(\theta)
                                               \end{array}\right)\Bigg].
\label{n1}
\end{eqnarray*}

By definition of the support function

\begin{eqnarray*}
  h_{\Omega_{\varepsilon}}(\varphi) &=& <\overrightarrow{OM},n>\\
             &=&  <R(\theta)\left(\begin{array}{c}
                                                 \cos(\theta) \\
                                                 \sin(\theta)
                                               \end{array}\right),\frac{1}{\sqrt{R^{2}(\theta)+R^{'2}(\theta)}}\Bigg[R(\theta)\left(\begin{array}{c}
                                                 \cos(\theta) \\
                                                 \sin(\theta)
                                               \end{array}\right)-R^{'}(\theta)\left(\begin{array}{c}
                                                 -\sin(\theta) \\
                                                 \cos(\theta)
                                               \end{array}\right)\Bigg]> \\
             &=& \frac{R^{2}(\theta)}{\sqrt{R^{'2}(\theta)+R^{2}(\theta)}}= R(\theta)\frac{1}{\sqrt{1+\frac{R^{'2}(\theta)}{R^{2}(\theta)}}} \\
             &=& R(\theta)\Bigg(1-\frac{1}{2}\frac{R^{'2}(\theta)}{R^{2}(\theta)}+O(\varepsilon^{3})\Bigg)= R(\theta)-\frac{1}{2}\frac{R^{'2}(\theta)}{R(\theta)}+O(\varepsilon^{3})\\
             &=& R(\theta)-\frac{1}{2}\frac{\Bigg(\varepsilon F^{'}(\theta)+\varepsilon^{2}G^{'}(\theta)+O(\varepsilon^{3})\Bigg)^{2}}{ 1+\varepsilon F(\theta)+\varepsilon^{2}G(\theta)+O(\varepsilon^{3})}+O(\varepsilon^{3})\\
             &=& R(\theta)-\frac{1}{2}\frac{\varepsilon^{2} (F^{'}(\theta))^{2}+O(\varepsilon^{3})}{ 1+\varepsilon F(\theta)+\varepsilon^{2}G(\theta)+O(\varepsilon^{3})}+O(\varepsilon^{3})\\
             &=& R(\theta)-\frac{1}{2}\Bigg(\varepsilon^{2} (F^{'}(\theta))^{2}+O(\varepsilon^{3})\Bigg)\Bigg(1-\varepsilon F(\theta)-\varepsilon^{2}G(\theta)+O(\varepsilon^{3})\Bigg)+O(\varepsilon^{3})\\
             &=& R(\theta)-\frac{1}{2}\varepsilon^{2} (F^{'}(\theta))^{2}+O(\varepsilon^{3})\\
             &=& 1+\varepsilon F(\theta)+\varepsilon^{2}G(\theta)+O(\varepsilon^{3})-\frac{1}{2}\varepsilon^{2} (F^{'}(\theta))^{2}+O(\varepsilon^{3})\\
             &=&1+\varepsilon F(\theta)+\varepsilon^{2}\Bigg(G(\theta)-\frac{1}{2}(F^{'}(\theta))^{2}\Bigg)+O(\varepsilon^{3}).
\end{eqnarray*}

According to the expression of $h_{\Omega_{\varepsilon}}$ in (\ref{hep}), we can conclude that
$
f(\varphi)=F(\theta)\; \mbox{ and }\; g(\varphi)=G(\theta)-\frac{1}{2}(F^{'}(\theta))^{2}.
$
By an asymptotic expansion of the function $f$ in the neighborhood of  $\theta$, we obtain:
\begin{eqnarray*}
  F(\theta) &=& f(\varphi)= f(\theta-\frac{R^{'}(\theta)}{R(\theta)})
            = f(\theta)- \frac{R^{'}(\theta)}{R(\theta)}f^{'}(\theta)+O(\varepsilon^{2}) \\
            &=& f(\theta)-\varepsilon F^{'}(\theta)f^{'}(\theta)+O(\varepsilon^{2}) \\
            &=& f(\theta)-\varepsilon (f^{'}(\theta))^{2}+O(\varepsilon^{2}).
\end{eqnarray*}

Using the same method as above, we obtain:
\begin{eqnarray*}
  G(\theta) &=& g(\varphi)+\frac{1}{2}(F^{'}(\theta))^{2}
            = g(\theta)+\frac{1}{2}(f^{'}(\theta))^{2}+O(\varepsilon^{2}).
\end{eqnarray*}

Replacing in  (\ref{eq.22}), we conclude
$$
       R(\theta,\varepsilon)=1+\varepsilon f(\theta)+\varepsilon^{2}\Bigg(g(\theta)-\frac{1}{2}(f^{'}(\theta))^{2}\Bigg)+O(\varepsilon^{3}).
$$
\end{proof}

\section{Asymptotic expansion of eigenvalues with respect to the radial deformation}\label{Dirichlet}
\label{sec:Dirichlet}
The aim of this section is to compute the eigenvalues of the Dirichlet Laplacian on a set of constant width with respect to the radial deformation of the unit disk.
Let us first recall the analytic expression of the eigenvalues and eigenfunctions of the Dirichlet Laplacian on the disk.

\begin{theorem}\label{theo9}
  The eigenvalues and eigenfunctions of the Dirichlet-Laplacian
  on the disk of radius $R$ (normalized for the $L^{2}$-norm) are given by
\begin{equation}\label{eq.2}
\begin{array}{lll}
\lambda=\frac{j_{0,p}^{2}}{R^{2}},
&
u(r,\theta)=\sqrt{\frac{1}{\pi}}\frac{1}{R|J_{0}^{'}(j_{0,p})|}J_{0}\Bigg(\frac{j_{0,p}r}{R}\Bigg)& p\geq 1
\\
\lambda=\frac{j_{m,p}^{2}}{R^{2}},
&
u(r,\theta)=\left\{
                    \begin{array}{ll}
                      \sqrt{\frac{2}{\pi}}\frac{1}{R|J_{m}^{'}(j_{m,p})|}J_{m}\Bigg(\frac{j_{m,p}r}{R}\Bigg) \cos(m\theta) \\
                     \sqrt{\frac{2}{\pi}}\frac{1}{R|J_{m}^{'}(j_{m,p})|}J_{m}\Bigg(\frac{j_{m,p}r}{R}\Bigg) \sin(m\theta)
                    \end{array}
                  \right.& m,p\geq 1
\end{array}
\end{equation}

 where $j_{m,p}$ is the $p-$th zero of the Bessel function $J_{m}$.
\end{theorem}
\begin{proof}
  See \cite{henrot2006}.
\end{proof}

We consider the deformation of the unit disk given in section \ref{sec:bodies}. For a given $\varepsilon \ge 0$, $\partial \Omega_{\varepsilon}$ is described by the following parametrization
\begin{equation}\label{eq.R}
R(\theta,\varepsilon)=1+\varepsilon \Bigg(\sum_{n=-\infty}^{\infty}a_{n}e^{in\theta}\Bigg)+ \varepsilon^{2}\Bigg(\sum_{n=-\infty}^{\infty}b_{n}e^{in\theta}-\frac{1}{2}\Bigg((\sum_{n=-\infty}^{\infty}a_{n}^{'}e^{in\theta})\Bigg)^{2}\Bigg)+O(\varepsilon^{3})
\end{equation}
Where $a_{n}^{'}=ina_{n}$, $a_{-n}=\overline{a_{n}}$ and $b_{-n}=\overline{b_{n}}$ for all $n$.

Let us consider an eigenvalue $\lambda$ of the disk from theorem \ref{theo9}. We know that there exist $m\geq 0$ and $p>0$ such that $\lambda=j_{m,p}^{2}$. So let as fix them. Now, part VII 6.5 of \cite{kato2013}, pp. 423-426, gives us an expression of the eigenvalues and eigenfunctions of the Laplacian on the new domains. For some more details we can also refer to the \cite{micheletti1972},pp. 155-160 and to \cite{nagy1946} for details on the theorem used in \cite{micheletti1972}. For the following and for the simplicity, let us denoted $\lambda(\Omega_{\varepsilon})=\omega^{2}$ and $u(r,\theta,\varepsilon)$ an associated eigenfunction. Note that even if it is not explicit, they depends on $m$ and $p$. Since the eigenfunctions given in (\ref{eq.2}) define a basis and since $J_{-n}=(-1)^{n}J_{n}$, $\forall n$, let us write
\begin{equation}\label{eq.3}
u(r,\theta,\varepsilon)=\sum_{-\infty}^{+\infty} A_{n}(\varepsilon)J_{n}(\omega r)e^{in\theta},\;\mbox{ such that }A_{-n}=(-1)^{n}\bar{A}_{n}
\end{equation}
and $A_{n}(\varepsilon)=\delta_{|n|m}\alpha_{n}+\varepsilon \beta_{n}+\varepsilon^{2}\gamma_{n}+O(\varepsilon^{3})$ where $\delta_{|n|m}$ is the Kronecker symbol.
We deduce that $\displaystyle \alpha_{-n}=(-1)^{n}\bar{\alpha}_{n},$ $\displaystyle\beta_{-n}=(-1)^{n}\bar{\beta}_{n}$ and $\displaystyle\gamma_{-n}=(-1)^{n}\bar{\gamma}_{n}$.

For $m\neq 0$, we have $\alpha_{m}\neq 0$ and
\begin{equation}\label{eq.5}
u(r,\theta,0)=\Bigg(\alpha_{m}e^{im\theta}+\overline{\alpha_{m}e^{im\theta}}\Bigg)J_{m}(\omega_0 r),
\end{equation}
$u(r,\theta,0)$ being an eigenfunction on the disk associated with $\lambda (\Omega_{0})=j_{m,p}^{2}$. Thus, if we choose $\alpha_{m}=1$, the eigenfunction associated to $j_{m,p}^{2}$ is $u_{m,p}=2J_{m}(j_{m,p}r)\cos(m\theta)$ and if we choose $\alpha_{m}=i$, $u_{m,p}=-2J_{m}(j_{m,p}r)\sin(m\theta)$.\\

A result of \cite{nagy1946}  shows that $\omega$ can be written as:
\begin{equation}\label{eq.6}
\omega=\omega_{0}+\varepsilon\omega_{1}+\varepsilon^{2} \omega_{2}+O(\varepsilon^{3}).
\end{equation}

Since $\lambda (\Omega_{\varepsilon})=\omega^{2}$, we obtain
\begin{equation}\label{eq.77}
\lambda(\Omega_{\varepsilon})=\omega_{0}^{2}+2\varepsilon \omega_{0}\omega_{1}+\varepsilon^{2}(2\omega_{0}\omega_{2}+\omega_{1}^{2})+O(\varepsilon^{3}).
\end{equation}

The Dirichlet boundary condition becomes
\begin{equation}\label{eq.8}
u(R(\theta,\varepsilon),\theta,\varepsilon)=\sum_{-\infty}^{+\infty}A_{n}(\varepsilon) J_{n}(\omega R(\theta,\varepsilon))e^{in\theta}=0.
\end{equation}

We remark that $\omega R=\omega_{0}+(\omega-\omega_{0})+\omega(R-1)$. Using this in (\ref{eq.8}) and expanding $J_{n}$ in a Taylor series, we obtain
\begin{eqnarray*}
0&=&\displaystyle \sum_{-\infty}^{+\infty}A_{n}(\varepsilon)\Bigg[J_{n}(\omega_{0})+J_{n}^{'}(\omega_{0})(\omega-\omega_{0}+\omega(R-1))\\
&&\qquad\qquad\qquad+\frac{1}{2}J_{n}^{''}(\omega_{0})(\omega-\omega_{0}+\omega(R-1))^{2}+O(\varepsilon^{3})\Bigg]e^{in\theta}.
\end{eqnarray*}

Using (\ref{eq.R}) for $R$, (\ref{eq.3}) for $A_{n}$ and (\ref{eq.6}) for $\omega$ in the previous equality, we obtain the equation
\begin{equation*}
\begin{array}{ll}
0= \sum_{-\infty}^{+\infty} (\delta_{|n|,m}\alpha_{n}+\varepsilon \beta_{n}+\varepsilon^{2}\gamma_{n})\Bigg(J_{n}(\omega_{0})\\
+J_{n}^{'}(\omega_{0})\Bigg[\varepsilon \omega_{1}+\varepsilon^{2} \omega_{2}+(\omega_{0}+\varepsilon \omega_{1}+\varepsilon^{2} \omega_{2})\Bigg(\varepsilon f(\theta)+\varepsilon^{2}(g(\theta)-\frac{1}{2}f^{'2}(\theta))\Bigg)\Bigg]\\
+\frac{1}{2}J_{n}^{''}(\omega_{0})\Bigg[\varepsilon \omega_{1}+\varepsilon^{2} \omega_{2}+(\omega_{0}+\varepsilon \omega_{1}+\varepsilon^2\omega_2)\Bigg(\varepsilon f(\theta)+\varepsilon^{2}(g(\theta)-\frac{1}{2}f^{'2}(\theta))\Bigg)\Bigg]^{2} \Bigg)e^{in\theta}+O(\varepsilon^{3})
\end{array}
\end{equation*}
Therefore,
\begin{eqnarray*}
  0 &=& \sum_{n}\delta_{|n|,m}\alpha_{n}J_{n}(\omega_{0})e^{in\theta} \\
    &+& \varepsilon \sum_{n}\Bigg(\beta_{n}J_{n}(\omega_{0})+\delta_{|n|,m}\alpha_{n}J_{n}^{'}(\omega_{0})\Bigg[\omega_{1}+\omega_{0}\sum_{l}a_{l}e^{il\theta}\Bigg]\Bigg)e^{in\theta} \\
    &+& \varepsilon^{2}\sum_{n}\Bigg(\gamma_{n}J_{n}(\omega_{0})+\beta_{n}J_{n}^{'}(\omega_{0})\Bigg[\omega_{1}+\omega_{0}\sum_{l}a_{l}e^{il\theta}\Bigg] \\
    && +\delta_{|n|,m}\alpha_{n}\Bigg[J_{n}^{'}(\omega_{0})\Bigg(\omega_{2}+\omega_{1}\sum_{l}a_{l}e^{il\theta}+\omega_{0}\Bigg(\sum_{l}b_{l}e^{il\theta}-\frac{1}{2}(\sum_{l}a_{l}e^{il\theta})^{'2}\Bigg)\Bigg)\\
    &+&  \frac{1}{2}J_{n}^{''}(\omega_{0})\Bigg[\omega_{1}^{2}+2\omega_{0}\omega_{1}\sum_{l}a_{l}e^{il\theta}+\omega_{0}^{2}(\sum_{l}a_{l}e^{il\theta})^{2}\Bigg]\Bigg]\Bigg)e^{in\theta}+O(\varepsilon^{3}).
\end{eqnarray*}
\newcommand{\Jn}{J_{n}(\omega_{0})}

This equation holds true if and only if the coefficients ahead of $\displaystyle\varepsilon^{j}, j=0,1,2 $ are all equal to zero.

We now have to separate the cases $m=0$, i.e. simple eigenvalues of the disk, and $m>0$, double eigenvalues, and we express $\omega_0$, $\omega_1$ and $\omega_2$ in terms of $(a_n)$, $(b_n)$, $(\alpha_n)$, $(\beta_n)\, \cdots$.

\subsection{Case $m=0$: simple eigenvalues}
We give the expression of simple eigenvalue of Dirichlet Laplacian among sets of constant width near to the disk, by assuming that $a_{l}=0$ for $l$ odd.
\begin{lemma}\label{Simple}
With the previous notations, if $\lambda(\Omega_{0})=j_{0,p}^{2}$ then
\begin{equation}\label{eq.R1}
  \lambda(\Omega_{\varepsilon})= j_{0,p}^{2}\Bigg(1+4\varepsilon^{2}\sum_{l\in \mathbb{N}^{*}}\Bigg(\frac{1}{2}+\frac{1}{2} l^{2}+j_{0,p}\frac{J_{l}^{'}(j_{0,p})}{J_{l}(j_{0,p})}\Bigg)|a_{l}|^{2}+O(\varepsilon^{3})\Bigg)
\end{equation}

\end{lemma}

\begin{proof}

\textbf{Term in $\varepsilon^{0}$}\\
$0=\alpha_{0}J_{0}(\omega_{0})$, and as $\alpha_{0}\neq 0$ thus $\omega_{0}=j_{0,p}$.

\textbf{Term in $\varepsilon^{1}$}
\begin{eqnarray*}
0&=&\alpha_{0}J_{0}^{'}(\omega_{0})\Bigg[\omega_{1}+\omega_{0}\Bigg(\sum_{l}a_{l}e^{il\theta}\Bigg)\Bigg]+\sum_{n\neq0}\beta_{n}J_{n}(\omega_{0})e^{in\theta}\\
&=&\alpha_{0}\omega_{1}J_{0}^{'}(\omega_{0})+\sum_{n\neq0}(\alpha_{0}J_{0}^{'}(\omega_{0})\omega_{0}a_{n}+\beta_{n}J_{n}(\omega_{0}))e^{in\theta}.
\end{eqnarray*}

Thus $\alpha_{0}J_{0}^{'}(\omega_{0})\omega_{0}a_{n}+\beta_{n}J_{n}(\omega_{0})=0$ for $n\neq 0$, and $\alpha_{0}\omega_{1}J_{0}^{'}(\omega_{0})=0$, so
$ \beta_{n}=-\alpha_{0}\frac{J_{0}^{'}(\omega_{0})}{J_{n}(\omega_{0})}\omega_{0}a_{n},\mbox{ for } n\neq 0$
and $\omega_{1}=0$.

\textbf{Term in $\varepsilon^{2}$}\\
Using $\omega_{1}=0$ and $J_{0}(\omega_{0})=0$
\begin{eqnarray*}
  0 &=& \sum_{n\neq 0} \Bigg(\gamma_{n}J_{n}(\omega_{0})+ \beta_{n}J_{n}^{'}(\omega_{0})\omega_{0}\sum_{l}a_{l}e^{il\theta}\Bigg)e^{in\theta}+\gamma_{0}J_{0}(\omega_{0})+ \beta_{0}J_{0}^{'}(\omega_{0})\omega_{0}\sum_{l}a_{l}e^{il\theta}\\
    &+&  \alpha_{0}\Bigg[J_{0}^{'}(\omega_{0})\Bigg(\omega_{2}+\omega_{0}\Bigg(\sum_{l}b_{l}e^{il\theta}-\frac{1}{2}\Bigg(\Bigg(\sum_{l}a_{l}e^{il\theta}\Bigg)^{'}\Bigg)^{2}\Bigg)\Bigg)\\
    &+&  \frac{1}{2}J_{n}^{''}(\omega_{0})\omega_{0}^{2}\sum_{l,k}a_{l}a_{n}e^{i(l+k)\theta}\Bigg]\\
    &=& \sum_{n\neq 0}\Bigg(\gamma_{n}J_{n}(\omega_{0})+\beta_{n}J_{n}^{'}(\omega_{0})\omega_{0}\sum_{l}a_{l}e^{il\theta}\Bigg)e^{in\theta}+\beta_{0}J_{0}^{'}(\omega_{0})\omega_{0}\sum_{l}a_{l}e^{il\theta}\\
    &+&\alpha_{0}\Bigg[J_{0}^{'}(\omega_{0})\Bigg(\omega_{2}+\omega_{0}\Bigg(\sum_{l}b_{l}e^{il\theta}-\frac{1}{2}\sum_{l,k}a_{l}^{'}a_{k}^{'}e^{i(l+k)\theta}\Bigg)\Bigg)\\
    &+&\frac{1}{2}J_{0}^{''}(\omega_{0})\omega_{0}^{2}(\sum_{l,k}a_{l}a_{k}e^{i(l+k)\theta})\Bigg].
\end{eqnarray*}

Then, by assumption $1$:

$$\sum_{n\neq0}\beta_{n}J_{n}^{'}(\omega_{0})\omega_{0}a_{-n}+\underbrace{\beta_{0}J_{0}^{'}(\omega_{0})\omega_{0}a_{0}}_{=0}+\alpha_{0}J_{0}^{'}(\omega_{0})\omega_{2}+\underbrace{\alpha_{0}b_{0}\omega_{0}J_{0}^{'}(\omega_{0})}_{=0}$$
$$-\frac{1}{2}\alpha_{0}J_{0}^{'}(\omega_{0})\omega_{0}\sum_{l}a_{l}^{'}a_{-l}^{'} + \frac{1}{2}\alpha_{0}\omega_{0}^{2}J_{0}^{''}(\omega_{0})\sum_{l} a_{l} a_{-l}=0.$$

Therefore,
\begin{eqnarray*}
  \omega_{2} &=& \frac{1}{2}\omega_{0}\sum_{l}|a_{l}|^{2} -\frac{1}{2}\omega_{0}^{2} \frac{J_{0}^{''}(\omega_{0})}{J_{0}^{'}(\omega_{0})}\sum_{l}|a_{l}^{'}|^{2} - \sum_{n\neq 0}\frac{\beta_{n}}{\alpha_{0}}a_{-n} \frac{J_{n}^{'}(\omega_{0})}{J_{0}^{'}(\omega_{0})}\omega_{0}.
\end{eqnarray*}

Using $\beta_{n}=-\alpha_{0}\frac{J_{0}^{'}(\omega_{0})}{J_{n}(\omega_{0})}\omega_{0}a_{n}$ and $J_{0}^{''}(\omega_{0})=\frac{-J_{0}^{'}(\omega_{0})}{\omega_{0}}$, we obtain

\begin{eqnarray*}
 \omega_{2}  &=& \frac{1}{2}\omega_{0}\sum_{l}|a_{l}^{'}|^{2}+\frac{1}{2}\omega_{0} \sum_{l}|a_{l}^{'}|^{2}+\sum_{n\neq 0}|a_{n}|^{2} \frac{J_{n}^{'}(\omega_{0})}{J_{n}(\omega_{0})}\omega_{0}^{2} \\
             &=&\omega_{0}\sum_{n\neq 0}\Bigg(\frac{1}{2}n^{2}+\omega_{0}\frac{J_{n}^{'}(\omega_{0})}{J_{n}(\omega_{0})}\Bigg)|a_{n}|^{2} \\
             &=& 2 \omega_{0}\sum_{n>0} \Bigg(\frac{1}{2}+\frac{1}{2}n^{2}+\omega_{0}\frac{J_{n}^{'}(\omega_{0})}{J_{n}(\omega_{0})}\Bigg)|a_{n}|^{2}.
\end{eqnarray*}
In conclusion, replacing $\omega_{0}$, $\omega_{1}$ and $\omega_{2}$ by these values in (\ref{eq.77}) we deduce (\ref{eq.R1}).
\end{proof}

\subsection{Case $m>0$: double eigenvalues}
We give the expression of double eigenvalue of Dirichlet Laplacian among set of constant width near to the disk, by assuming that $a_{l}=0$ for $l$ odd.
\begin{lemma}\label{Double}
With previous notations, if $\lambda (\Omega_{0})=j_{m,p}^{2}$, for $m>0$, then
\begin{eqnarray*}
\lambda (\Omega_{\varepsilon})   &=& j_{m,p}^{2}\Bigg[1+2\varepsilon^{2}\Bigg(\sum_{|l|\neq m}\Bigg(\frac{1}{2}+\frac{1}{2}(m-l)^{2}+j_{m,p}\frac{J_{l}^{'}(j_{m,p})}{J_{l}(j_{m,p})}\Bigg)|a_{m-l}|^{2} \\
                                 && \quad\quad + \frac{\alpha_{m}}{\overline{\alpha_{m}}}\Bigg(\sum_{l\neq m}\frac{1}{2}-\frac{1}{2}(m^{2}-l^{2})+j_{m,p}\frac{J_{l}^{'}(j_{m,p})}{J_{l}(j_{m,p})}\Bigg)a_{m+l}a_{m-l}\Bigg)+O(\varepsilon^{3})\Bigg]\\
                                 &=& j_{m,p}^{2}\Bigg[1+2\varepsilon^{2}\Bigg(\sum_{|l|\neq m}\Bigg(\frac{1}{2}+\frac{1}{2}(m-l)^{2}+j_{m,p}\frac{J_{l}^{'}(j_{m,p})}{J_{l}(j_{m,p})}\Bigg)|a_{m-l}|^{2} \\
                                 &&\quad\quad \pm \Bigg|\Bigg(\sum_{l\neq m}\frac{1}{2}-\frac{1}{2}(m^{2}-l^{2})+j_{m,p}\frac{J_{l}^{'}(j_{m,p})}{J_{l}(j_{m,p})}\Bigg)a_{m+l}a_{m-l}\Bigg|\Bigg)+O(\varepsilon^{3})\Bigg].
\end{eqnarray*}
\end{lemma}
\begin{proof}
\textbf{Term in $\varepsilon^{0}$}\\
As $\alpha_{-m}=(-1)^m\overline{\alpha_m}\not=0$ and $J_{-m}=(-1)^mJ_m$ then
$$
\displaystyle\alpha_{m}J_{m}(\omega_{0})e^{im\theta}+\alpha_{-m}J_{-m}(\omega_{0})e^{-im\theta}=2Re\Bigg(\alpha_{m}e^{im\theta}\Bigg)J_{m}(\omega_{0})=0,
$$
so $J_{m}(\omega_{0})=0$ implies $\omega_{0}=j_{m,p}$.

\textbf{Term in $\varepsilon^{1}$}
%

\begin{eqnarray*}
\displaystyle0   &=& \sum_{|n|\neq m}\beta_{n}J_{n}(\omega_{0})e^{in\theta}+ \underbrace{\beta_{m}J_{m}(\omega_{0})e^{im\theta}}_{=0}+\underbrace{\beta_{-m}J_{-m}(\omega_{0})e^{-im\theta}}_{=0}\\
                 && \qquad\qquad\qquad\quad\quad+\Bigg(\alpha_{m} J_{m}^{'}(\omega_{0})e^{im\theta}+\alpha_{-m} J_{-m}^{'}(\omega_{0})e^{-im\theta}\Bigg)\Bigg[\omega_{1}+\omega_{0}\sum_{l}a_{l}e^{il\theta}\Bigg] \\
                 &=& \sum_{|n|\neq m}\beta_{n}J_{n}(\omega_{0})e^{in\theta}+ J_{m}^{'}(\omega_{0})\Bigg(\alpha_{m}e^{im\theta}+\overline{\alpha_{m}}e^{-im\theta}\Bigg)\omega_{1} \\
                 &&  \qquad\qquad\qquad\quad\quad+\omega_{0}J_{m}^{'}(\omega_{0})\Bigg(\alpha_{m}\sum_{l}a_{l}e^{i(l+m)\theta}+\overline{\alpha_{m}}\sum_{l}a_{l}e^{i(l-m)\theta}\Bigg)\\
                 &=& \sum_{|n|\neq m}\Bigg[\beta_{n}J_{n}(\omega_{0})+\Bigg(\alpha_{m}a_{n-m}+ \overline{\alpha_{m}}a_{n+m}\Bigg)\omega_{0}J_{m}^{'}(\omega_{0})\Bigg]e^{in\theta} \\
                 &&\quad+ \Bigg(\alpha_{m}\omega_{1}+\overline{\alpha_{m}}\omega_{0}a_{2m}\Bigg )J_{m}^{'}(\omega_{0})e^{im\theta}+\Bigg(\overline{\alpha_{m}}\omega_{1}+\alpha_{m}\omega_{0}a_{-2m}\Bigg )J_{m}^{'}(\omega_{0})e^{-im\theta}.
\end{eqnarray*}

For $j\neq m$, we have $
\beta_{j}=-\omega_{0}(\alpha_{m}a_{j-m}+\overline{\alpha_{m}}a_{j+m})\frac{J_{m}^{'}(\omega_{0})}{J_{j}(\omega_{0})}.
$

For $j= m$, we have $
(\alpha_{m}\omega_{1}+\overline{\alpha_{m}}\omega_{0}a_{2m})J_{m}^{'}(\omega_{0})=0.
$ Since $a_{2m}=0$, $\alpha_{m}\neq 0$ and $J_{m}^{'}(\omega_{0})\neq 0$, $\omega_{1}=0$.\\

\textbf{Term in $\varepsilon^{2}$}

By using $\omega_{1}=0$:\\
$
0 = \sum_{n}\bigg[\gamma_{n}J_{n}(\omega_{0}) +\beta_{n}J_{n}^{'}(\omega_{0})\omega_{0}(\sum_{l}a_{l}e^{il\theta})\Bigg]e^{in\theta}
$

$$
\quad\quad+\alpha_{m}\Bigg[ J_{m}^{'}(\omega_{0})\Bigg(\omega_{2}+\omega_{0}\Bigg(\sum_{l}b_{l}e^{il\theta}-\frac{1}{2}\Bigg(\Bigg(\sum_{n}a_{n}e^{in\theta}\Bigg)^{'}\Bigg)^{2}\Bigg)\Bigg)
$$
$$
\quad\quad+\frac{1}{2}J_{m}^{''}(\omega_{0})\omega_{0}^{2}\Bigg(\sum_{l}a_{l}e^{il\theta}\Bigg)^{2}\Bigg]e^{im\theta} $$
$$
\quad\quad+\alpha_{-m}\Bigg[ J_{-m}^{'}(\omega_{0})\Bigg(\omega_{2}+\omega_{0} \Bigg(\sum_{l}b_{l}e^{il\theta}-\frac{1}{2}\Bigg(\Bigg(\sum_{n}a_{n}e^{in\theta}\Bigg)^{'}\Bigg)^{2}\Bigg)\Bigg)
$$
$$
\quad\quad+\frac{1}{2}J_{-m}^{''}(\omega_{0})\omega_{0}^{2}\Bigg(\sum_{l}a_{l}e^{il\theta}\Bigg)^{2}\Bigg]e^{-im\theta}.
$$


Since $J_{m}^{''} (\omega_{0})=\frac{-1}{\omega_{0}}J_{m}^{'} (\omega_{0})$,$J_{-m}=(-1)^{m}J_{m}$ and $\alpha_{m}=(-1)^{m}\overline{\alpha_{m}}$, we have
\begin{equation*}
\begin{array}{ll}
0 =\sum\limits_{n}\Bigg(\gamma_{n}J_{n}(\omega_{0}) +\beta_{n}J_{n}^{'}(\omega_{0})\omega_{0}\sum\limits_{l}a_{l}e^{il\theta}\Bigg) e^{in\theta}
\\
\qquad+ \alpha_{m}J_{m}^{'}(\omega_{0})\Bigg(\omega_{2}+\omega_{0}\Bigg(\sum\limits_{l}b_{l}e^{il\theta}-\frac{1}{2} \sum\limits_{n,l}a_{l}^{'}a_{n}^{'}e^{i(l+n)\theta}-\frac{1}{2}\sum\limits_{n,l}a_{n}a_{l}e^{i(l+m)\theta}\Bigg)\Bigg)e^{im\theta}
\\
\qquad+ \overline{\alpha_{m}}J_{m}^{'}(\omega_{0})\Bigg(\omega_{2}+\omega_{0}\Bigg(\sum\limits_{l}b_{l}e^{il\theta}-\frac{1}{2} \sum\limits_{n,l}a_{l}^{'}a_{n}^{'}e^{i(l+n)\theta}-\frac{1}{2}\sum\limits_{n,l}a_{n}a_{l}e^{i(l+m)\theta}\Bigg)\Bigg)e^{-im\theta}.
\end{array}
\end{equation*}

For the $m$-th coefficient, we have $J_{m}(\omega_{0})=0$  and $a_{l}^{'}=i\;l\;a_{l}$, then
\begin{equation*}
\begin{array}{ll}
\displaystyle 0 = \omega_{0}\sum_{l}\beta_{l}a_{m-l}J_{l}^{'}(\omega_{0}) +\alpha_{m}J_{m}^{'}(\omega_{0})\omega_{2}+\underbrace{\alpha_{m}J_{m}^{'}(\omega_{0})\omega_{0}b_{0}}_{=0}
\\
\displaystyle\quad\quad-\frac{1}{2}\omega_{0}\alpha_{m}J_{m}^{'}(\omega_{0})\sum_{l}|a'_{l}|^{2}
-\frac{1}{2}\alpha_{m}J_{m}^{'}(\omega_{0})\omega_{0}\sum_{l}|a_{l}|^{2}
\\
\displaystyle\quad\quad+\underbrace{\overline{\alpha_{m}}J_{m}^{'}(\omega_{0})\omega_{0}b_{2m}}_{=0}-\frac{1}{2}\overline{\alpha_{m}}J_{m}^{'}(\omega_{0})\omega_{0}\sum_{l\in \mathbb{Z}}a_{m+l}^{'}a_{m-l}^{'}-\frac{1}{2}\overline{\alpha_{m}}J_{m}^{'}(\omega_{0})\omega_{0}\sum_{l}a_{m+l}a_{m-l}
\end{array}
\end{equation*}

We have that
\begin{equation*}
\begin{array}{lll}
  \omega_{2}&=&-\frac{\omega_{0}}{\alpha_{m}}\sum_{|l|\neq m}\beta_{l}\frac{J_{l}^{'}(\omega_{0})}{J_{m}(\omega_{0})}a_{m-l}+\underbrace{\beta_{m}J_{m}^{'}(\omega_{0})a_{0}}_{0}+\underbrace{\beta_{-m}J_{-m}^{'}(\omega_{0})a_{2m}}_{0}+\frac{1}{2}\omega_{0}\sum_{l}|a_{l}^{'}|^{2}+\frac{1}{2}\omega_{0}\sum_{l}|a_{l}|^{2}
\\
&&+\frac{\overline{\alpha_{m}}}{\alpha_{m}}\omega_{0}\Bigg(\frac{1}{2}\sum_{l\in \mathbb{Z}}a_{m+l}^{'}a_{m-l}^{'}+\frac{1}{2}\omega_{0}\sum_{l}a_{m+l}a_{m-l}\Bigg).
\end{array}
\end{equation*}

Using the previous expression of $\beta_{j}$,$j\neq m$ and with $a_{0}$, $a_{2m}=0$, we deduce that
\begin{equation*}
  \frac{\omega_{0}}{\alpha_{m}}\sum_{|l|\neq m}\beta_{l}\frac{J_{l}^{'}(\omega_{0})}{J_{m}^{'}(\omega_{0})}a_{m-l} = -\omega_{0}^{2}\sum_{l\neq m}a_{l-m}a_{m-l}\frac{J_{l}^{'}(\omega_{0})}{J_{l}(\omega_{0}}                                                                                                                    -\omega_{0}^{2}\frac{\overline{\alpha_{m}}}{\alpha_{m}}\sum_{l\neq m}a_{m+l}a_{m-l}\frac{J_{l}^{'}(\omega_{0})}{J_{l}(\omega_{0}}.
\end{equation*}

Thus,
\begin{eqnarray*}
  \omega_{2}                    &=& \omega_{0}^{2}\sum_{l}\underbrace{a_{l-m}a_{m-l}}_{=|a_{m-l}|^{2}}\frac{J_{l}^{'}(\omega_{0})}{J_{l}(\omega_{0})}+\omega_{0}^{2}\frac{\overline{\alpha_{m}}}{\alpha_{m}}\sum_{l}a_{l+m}a_{m-l}\frac{J_{l}^{'}(\omega_{0})}{J_{l}(\omega_{0})} \\
                                &&  +\frac{1}{2}\omega_{0}\sum_{l}|a_{l}^{'}|^{2}+\frac{1}{2}\omega_{0}\sum_{l}|a_{l}|^{2}+\frac{\overline{\alpha_{m}}}{\alpha_{m}}\omega_{0}\Bigg[\frac{1}{2}\sum_{l}a_{m+l}^{'}a_{m-l}^{'}+\frac{1}{2}\sum_{|l|\neq m}a_{m-l}a_{m+l}\Bigg]
\end{eqnarray*}
so,
\begin{eqnarray*}
  \frac{\omega_{2}}{\omega_{0}} &=& \frac{1}{2}\sum_{l} |a_{l}|^{2}+\frac{1}{2}\sum_{l}|a_{l}^{'}|^{2}+\omega_{0}\sum_{|l|\neq m}|a_{m-l}|^{2}\frac{J_{l}^{'}(\omega_{0})}{J_{l}(\omega_{0}}\\
                                && + \frac{\overline{\alpha_{m}}}{\alpha_{m}}\Bigg( \frac{1}{2}\sum_{|l|\neq m}a_{m-l}a_{m+l}+\frac{1}{2}\sum_{|l|\neq m}a_{m-l}^{'}a_{m+l}^{'}+\omega_{0}\sum_{|l|\neq m}a_{l+m}a_{m-l}\frac{J_{l}^{'}(\omega_{0})}{J_{l}(\omega_{0}}\Bigg).
\end{eqnarray*}

We have that
$
\displaystyle\sum_{l\in \mathbb{Z}}|a_{l}|^{2}=\underbrace{|a_{0}|^{2}}_{=0}+\underbrace{|a_{2m}|^{2}}_{=0}+\sum_{l\neq m}|a_{m-l}|^{2}$,  $\displaystyle\sum_{l\in \mathbb{Z}}|a_{l}^{'}|^{2}=\underbrace{|a_{0}^{'}|^{2}}_{=0}+\underbrace{|a_{2m}^{'}|^{2}}_{=0}+\sum_{l\neq m}|a_{m-l}^{'}|^{2}$, $a_{-n}^{'}=-i n a_{-n}$ and $a_{-n}^{'}a_{n}^{'}=n^{2}a_{-n}a_{n}=n^{2}|a_{n}|^{2}$.

Finally we have:
\begin{equation*}
\begin{array}{lll}
\displaystyle\frac{\omega_{2}}{\omega_{0}}&=&\underbrace{\sum_{|l|\neq m}\Bigg(\frac{1}{2}+\frac{1}{2}(m-l)^{2}+\omega_{0}\frac{J_{l}^{'}(\omega_{0})}{J_{l}(\omega_{0})}\Bigg)|a_{m-l}|^{2}}_{\Gamma}
\\
&&+\underbrace{\frac{\alpha_{m}}{\overline{\alpha_{m}}}\Bigg(\sum_{l\neq m}\frac{1}{2}-\frac{1}{2}(m^{2}-l^{2})+\omega_{0}\frac{J_{l}^{'}(\omega_{0})}{J_{l}(\omega_{0})}\Bigg)a_{m+l}a_{m-l}}_{\Upsilon}.
\end{array}
\end{equation*}
Similarly for $-m$,
\begin{eqnarray*}
  0 &=& \omega_{0}\sum_{l\in \mathbb{Z}}\beta_{l}J_{l}^{'}(\omega_{0})a_{-m-l}+\overline{\alpha_{m}}J_{m}^{'}(\omega_{0})\Bigg(\omega_{2}+\omega_{0}b_{0}-\frac{1}{2}\omega_{0}\sum_{l}|a_{l}^{'}|^{2}-\frac{\omega_{0}}{2}\sum_{l}|a_{l}|^{2}\Bigg)\\
    &&\quad\quad+\alpha_{m}J_{m}^{'}(\omega_{0})\Bigg(\omega_{0}b_{-2m}-\frac{1}{2}\omega_{0}\sum_{l\in \mathbb{Z}}a_{-m+l}^{'}a_{-m-l}^{'}-\frac{\omega_{0}}{2}\sum_{l\in \mathbb{Z}}a_{-m+l}a_{-m-l}\Bigg)\\
    &=& \omega_{0}\sum_{l\in \mathbb{Z}}\beta_{l}J_{l}^{'}(\omega_{0})a_{-m-l}+\overline{\alpha_{m}}J_{m}^{'}(\omega_{0})\omega_{2}+\overline{\alpha_{m}}J_{m}^{'}(\omega_{0})\omega_{0}\Bigg(-\frac{1}{2}\sum_{l}|a_{l}^{'}|^{2}-\frac{1}{2}\sum_{l}|a_{l}|^{2}\Bigg) \\
    &&\quad\quad+\alpha_{m}J_{m}^{'}(\omega_{0})\omega_{0}\Bigg(-\frac{1}{2}\sum_{l}a_{-m+l}^{'}a_{-m-l}^{'}-\frac{1}{2}\sum_{l}a_{-m+l}a_{-m-l}\Bigg).
  \end{eqnarray*}

By using the previous expression of $\beta_{l}$, we deduce
\begin{eqnarray*}
 \frac{\omega_{2}}{\omega_{0}}  &=& \frac{1}{2}\sum_{l\in \mathbb{Z}}|a_{l}|^{2}+\frac{1}{2}\sum_{l\in \mathbb{Z}}|a_{l}^{'}|^{2}+\omega_{0}\Bigg(\sum_{l\neq m}\frac{\alpha_{m}}{\overline{\alpha_{m}}}a_{l-m}a_{-m-l}\frac{J_{l}^{'}(\omega_{0})}{J_{l}(\omega_{0})}+a_{l+m}a_{-m-l}\frac{J_{l}^{'}(\omega_{0})}{J_{l}(\omega_{0})}\Bigg)  \\
                                &&\quad\quad+\frac{\alpha_{m}}{\overline{\alpha_{m}}}\Bigg(\frac{1}{2}\sum_{l}a_{-m+l}a_{-m-l}+ \frac{1}{2}\sum_{l}a_{-m+l}^{'}a_{-m-l}^{'}\Bigg)\\
                                &=& \frac{1}{2}\sum_{l\in \mathbb{Z}}|a_{l}|^{2}+ \frac{1}{2}\sum_{l\in \mathbb{Z}}|a_{l}^{'}|^{2}+\omega_{0}\sum_{|l|\neq m} |a_{m+l}|^{2}\frac{J_{l}^{'}(\omega_{0})}{J_{l}(\omega_{0})}\\
                                &&\quad\quad+\frac{\alpha_{m}}{\overline{\alpha_{m}}}\Bigg[\frac{1}{2}\sum_{|l|\neq m} \overline{a_{m-l}}\overline{a_{m+l}}+\frac{1}{2}\sum_{|l|\neq m} \overline{a_{m-l}^{'}}\overline{a_{m+l}^{'}}+\omega_{0}\sum_{|l|\neq m} (\overline{a_{m-l}^{'}}\overline{a_{m+l}^{'}})\frac{J_{l}^{'}(\omega_{0})}{J_{l}(\omega_{0})}\Bigg]
\end{eqnarray*}
and with
\begin{eqnarray*}
  \sum_{|l|\neq m}|a_{m+l}|^{2}\frac{J_{l}^{'}(\omega_{0})}{J_{l}(\omega_{0})}
&=&  \sum_{|l|\neq m}|a_{m-l}|^{2}\frac{J_{-l}^{'}(\omega_{0})}{J_{-l}(\omega_{0})}
=  \sum_{|l|\neq m}|a_{m-l}|^{2}\frac{J_{l}^{'}(\omega_{0})}{J_{l}(\omega_{0})}
\end{eqnarray*}
we have
\begin{equation*}
\begin{array}{lll}
\frac{\omega_{2}}{\omega_{0}}&=&\underbrace{\sum_{|l|\neq m}\Bigg(\frac{1}{2}-\frac{1}{2}(m-l)^{2}+\omega_{0}\frac{J_{l}^{'}(\omega_{0})}{J_{l}(\omega_{0})}\Bigg)|a_{m-l}|^{2}}_{\Gamma}
\\
&&+\underbrace{\frac{\alpha_{m}}{\overline{\alpha_{m}}}\sum_{l\neq m}\Bigg(\frac{1}{2}-\frac{1}{2}(m^{2}-l^{2})+\omega_{0}\frac{J_{l}^{'}(\omega_{0})}{J_{l}(\omega_{0})}\Bigg)\overline{a_{m+l}}\quad\overline{a_{m-l}}}_{\overline{\Upsilon}}.
\end{array}
\end{equation*}

Thus, since $\displaystyle\frac{\omega_{2}}{\omega_{0}}=\Gamma+\overline{\Upsilon}=\Gamma+\Upsilon=\overline{\Gamma+\Upsilon}$ so $\frac{\omega_{2}}{\omega_{0}}\in \mathbb{R}$ and $\omega_{2}\in \mathbb{R}$. Furthermore, $\Gamma\in \mathbb{R}$, so $\Upsilon\in \mathbb{R}$, and, in particular, $\Upsilon=\pm |\Upsilon|$.

In conclusion, replacing $\omega_{0}$, $\omega_{1}$, $\omega_{2}$ by these values in (\ref{eq.77}) we deduce lemma \ref{Double}.
\end{proof}

\section{Results about the local minimality of the disk in a smooth neighborhood}
\label{sec:Results}

Each $(a_{n})$ and $(b_{n})$ verifying the assumption \ref{asumption} correspond to a convex body of constant width denoted $\Omega_{\varepsilon}$. In the previous section, we obtained asymptotic development of the eigenvalues $\lambda(\Omega_{\varepsilon})$ with respect to a small deformation of the disk. We should now to compare $\lambda(\Omega_{\varepsilon})$ and $\lambda (\Omega_{0})$. Now, if we find some families $(a_{n})$ and $(b_{n})$  such that $\lambda(\Omega_{\varepsilon})< \lambda (\Omega_{0})$, then we can deduce that the disk is not a local minimizer for the corresponding eigenvalue. Otherwise, if for all $a_{n}$ and $b_{n}$ verifying the assumption \ref{asumption} we get $\lambda(\Omega_{\varepsilon})\geq \lambda (\Omega_{0})$, then the disk  is a local minimizer for the corresponding eigenvalue. Also, there are some eigenvalues for which we can not conclude if the disk is a local minimizer or not. The following theorem gives the value of $\kappa$ for which the disk is a local minimizer. We notice that the result for cases $\kappa=1$ and $\kappa=3$ is included in the result given in section \ref{firstmainresult}.

\begin{theorem}\label{Eigenvalues}
The disk is a local minimizer among sets of constant width for $\lambda_{\kappa}$, where $\kappa\in \{1,3,5,8,12,$ $17,27,34,42\}$.
\end{theorem}
\begin{proof}
To prove this result, we prove that $\lambda(\Omega_{\varepsilon})\geq \lambda (\Omega_{0})$ for all $(a_{n})$ and $(b_{n})$. For the linkage between $\lambda_{\kappa}$ and $j_{m,p}$, in the case of disk of radius $1$, see table 3 in Appendix D.

\begin{description}
  \item[$\bullet$] Case $\kappa=1$:\\
  For $\Omega_{\varepsilon}$ a convex set with constant width, we have computed the simple eigenvalues in lemma \ref{Simple}:

\begin{equation}\label{eq.R1}
  \lambda(\Omega_{\varepsilon})= j_{0,p}^{2}\Bigg(1+4\varepsilon^{2}\sum_{k\in \mathbb{N}^{*}}\Bigg(\frac{1}{2}+\frac{k^{2}}{2}+j_{0,p}\frac{J_{k}^{'}(j_{0,p})}{J_{k}(j_{0,p})}\Bigg)|a_{k}|^{2}+O(\varepsilon^{3})\Bigg)
\end{equation}

The first eigenvalue correspond to $p=1$. J. Landau in \cite{land1999}, p 194 gives a detailed picture of the graph of $F_{n}(x)=x\frac{J_{n}^{'}(x)}{J_{n}(x)}$. It decreases from $n$ at $x=0$ to $-\infty$ at $x=j_{n,1}$, jumping to $+\infty$ and decreases to $-\infty$ in each  interval $]j_{n,p}, j_{n,p+1}[$ for all natural number $p\geq 1$. So if $x\leq j_{n,1}^{'}$ then $F_{n}(x)\geq 0$. We have

\begin{equation*}\label{Landau}
  x\frac{J_{n}^{'}(x)}{J_{n}(x)}\geq 0, \quad \mbox{for}\quad 0 \leq x \leq j_{n,1}^{'}.
\end{equation*}

Using this inequality and $j_{0,1}\leq j_{k,1}^{'}\quad \forall k\geq 2$. We obtain

$$
\frac{1}{2}+\frac{k^{2}}{2}+j_{0,1}\frac{J_{k}^{'}(j_{0,1})}{J_{k}(j_{0,1})}\geq 0, \; \forall k\geq 2.
$$

For $k=1$, using the equality $j_{m,p}\frac{J_{m+1}^{'}(j_{m,p})}{J_{m+1}(j_{m,p})}=-(m+1)$ (see Appendix \ref{sec:Appendix}) for $m=0$ and $p=1$, we find that
$
\frac{1}{2}+\frac{1}{2}+j_{0,1}\frac{J_{1}^{'}(j_{0,1})}{J_{1}(j_{0,1})}=0.
$

So,we conclude that

$$
\frac{1}{2}+\frac{k^{2}}{2}+j_{0,1}\frac{J_{k}^{'}(j_{0,1})}{J_{k}(j_{0,1})}\geq  0 \;\mbox{for all odd naturel number k},
$$

that implies $\lambda(\Omega_{\varepsilon})\geq \lambda(\Omega_{0})=j_{0,1}^{2}$.\\

  \item[$\bullet$] Case $\kappa\in \{3,5,8,12,17,27,34,42\}$:

 In lemma \ref{Double}, we have proved that the double eigenvalue can be written as:
\begin{equation}\label{Ja.3}
\begin{array}{lll}
\lambda (\Omega_{\varepsilon})=j_{m,p}^{2}\Bigg[1+2\varepsilon^{2}\Bigg(\sum\limits_{|l|\neq m}\Bigg(\frac{1}{2}+\frac{(m-l)^{2}}{2}+j_{m,p}\frac{J_{l}^{'}(j_{m,p})}{J_{l}(j_{m,p})}\Bigg)|a_{m-l}|^{2}
\\
\quad\quad\pm\quad\Bigg|\sum\limits_{l\neq m}\Bigg(\frac{1}{2}-\frac{(m^{2}-l^{2})}{2}+j_{m,p}\frac{J_{l}^{'}(j_{m,p})}{J_{l}(j_{m,p})}\Bigg)a_{m+l}a_{m-l}\Bigg|\Bigg)+O(\varepsilon^{3})\Bigg]
\end{array}.
\end{equation}

Since, $\;|a_{m-l}|=|a_{l-m}|$, the first term in (\ref{Ja.3}) is

\begin{equation*}
\begin{array}{ll}
\sum\limits_{|l|\neq m}\Bigg(\frac{1}{2}+\frac{(m-l)^{2}}{2}+j_{m,p}\frac{J_{l}^{'}(j_{m,p})}{J_{l}(j_{m,p})}\Bigg)|a_{m-l}|^{2}.
\\
\quad=\displaystyle\sum\limits_{l\in \mathbb{N},|l|\neq m}\Bigg(1+(m-l)^{2}+\jnp\frac{J_{l}^{'}(\jnp)}{J_{l}(\jnp)}+\jnp\frac{J_{2m-l}^{'}(\jnp)}{J_{2m-l}(\jnp)}\Bigg)|a_{m-l}|^{2}
\end{array}
\end{equation*}
If we take $k=m-l$, then $l=m-k$ and $2m-l=m+k$
\begin{equation*}
\begin{array}{ll}
\sum\limits_{|l|\neq m}\Bigg(1+(m-l)^{2}+\jnp\frac{J_{l}^{'}(\jnp)}{J_{l}(\jnp)}+\jnp\frac{J_{2m-l}^{'}(\jnp)}{J_{2m-l}(\jnp)}\Bigg)|a_{m-l}|^{2}
\\
=\displaystyle\sum\limits_{ k\in \mathbb{N},k\neq0,k\neq2m}\Bigg(1+k^{2}+\jnp\frac{J_{k+m}^{'}(\jnp)}{J_{k+m}(\jnp)}+\jnp\frac{J_{k-m}^{'}(\jnp)}{J_{k-m}(\jnp)}\Bigg)|a_{k}|^{2}\\
=\displaystyle\sum\limits_{ k\in \mathbb{N},k\neq0,k\neq2m}C_{k,m}(j_{m,p})|a_{k}|^{2}.
\end{array}
\end{equation*}
where $C_{k,m}(j_{m,p})=1+k^{2}+\jnp\frac{J_{k+m}^{'}(\jnp)}{J_{k+m}(\jnp)}+\jnp\frac{J_{k-m}^{'}(\jnp)}{J_{k-m}(\jnp)}$.\\

In the following, we are going to prove that the disk is a local minimizer for $\kappa=3$, using the above lemma. For the other eigenvalues, the proof is similar.

\begin{lemma}\label{P}

  \begin{equation}
  C_{k,m}(j_{m,p})=1+k^{2}+\jnp\frac{J_{k+m}^{'}(\jnp)}{J_{k+m}(\jnp)}+\jnp\frac{J_{k-m}^{'}(\jnp)}{J_{k-m}(\jnp)}\geq 0
\end{equation}
for all $k$ odd natural number and $\jnp\in\Bigg\{j_{1,1}; j_{2,1}; j_{3,1}; j_{4,1}; j_{5,1}; j_{5,2}; j_{6,2}; j_{7,1}\Bigg\}$.
\end{lemma}
\begin{proof}
  See appendix \ref{sec:AppendixB2}.
\end{proof}
For $m=1$ and $p=1$, (\ref{Ja.3}) gives the eigenvalues

\begin{equation}\label{above}
  \begin{array}{ll}
\left(
    \begin{array}{c}
     \lambda_{3}(\Omega_{\varepsilon})  \\
      \lambda_{2}(\Omega_{\varepsilon}) \\
    \end{array}
  \right)=j_{1,1}^{2}\Bigg[1+2\varepsilon^{2}\Bigg(\displaystyle\sum\limits_{ k\in \mathbb{N},k\neq0,k\neq2m}C_{k,1}(j_{1,1})|a_{k}|^{2}
\\
\quad\quad\pm\quad\Bigg|\sum\limits_{l\neq m}\Bigg(\frac{1}{2}-\frac{(1^{2}-l^{2})}{2}+j_{1,1}\frac{J_{l}^{'}(j_{1,1})}{J_{l}(j_{1,1})}\Bigg)a_{1+l}a_{1-l}\Bigg|\Bigg)+O(\varepsilon^{3})\Bigg]
\end{array}.
\end{equation}

The lower sign gives $\lambda_{2}(\Omega_{\varepsilon})$ and  the upper sign gives $\lambda_{3}(\Omega_{\varepsilon})$. From lemma \ref{P}, we can see that
$C_{k,1}(j_{1,1})>0,\; \forall k\in \mathbb{N}^{*} \mbox{ odd},$
so, we conclude that the first term in (\ref{above}) is positive.

Likewise, the last term of (\ref{above}) is an absolute value. Therefore, the eigenvalue $\lambda_{3}(\Omega_{\varepsilon})$ (obtained by the upper sign) satisfies
$$
\lambda_{3}(\Omega_{\varepsilon})\geq j_{1,1}^{2}=\lambda_{3}(\Omega_{0}).
$$

\end{description}
\end{proof}
\begin{remark}
  The open cases come from the fact that we are not able to compute the sign of
  \begin{equation}\label{eq.R3}
\begin{array}{lll}
\lambda (\Omega_{\varepsilon})=j_{m,p}^{2}\Bigg[1+2\varepsilon^{2}\Bigg(\displaystyle\sum\limits_{ k\in \mathbb{N},k\neq0,k\neq2m}\Bigg(1+k^{2}+\jnp\frac{J_{k+m}^{'}(\jnp)}{J_{k+m}(\jnp)}+\jnp\frac{J_{k-m}^{'}(\jnp)}{J_{k-m}(\jnp)}\Bigg)|a_{k}|^{2}
\\
\quad\quad-\quad\Bigg|\sum\limits_{l\neq m}\Bigg(\frac{1}{2}-\frac{(m^{2}-l^{2})}{2}+j_{m,p}\frac{J_{l}^{'}(j_{m,p})}{J_{l}(j_{m,p})}\Bigg)a_{m+l}a_{m-l}\Bigg|\Bigg)+O(\varepsilon^{3})\Bigg]
\end{array}.
\end{equation}
for $\jnp\in\Bigg\{j_{1,1}; j_{2,1}; j_{3,1}; j_{4,1}; j_{5,1}; j_{5,2}; j_{6,2}; j_{7,1}\Bigg\}$. Using the second shape derivative, we can give a positive answer for these cases.
\end{remark}

\begin{theorem}\label{X}
The disk is not a local minimizer among sets of constant width for $\lambda_{\kappa}$, where $\kappa\in \mathbb{N}^{*}\setminus\{1,2,3,4,5,7,8,11,12,16,17,26,27,33,34,41,49,50\}$.
\end{theorem}

To prove this theorem, we will distinguish two cases: simple eigenvalues and double eigenvalues.

\begin{proposition}\label{Simp}
Let $\lambda_{\kappa}$ be an eigenvalue of the Dirichlet Laplacian which is simple for the disk. For $\kappa\neq 1$, $\lambda_{\kappa}$ is not locally minimized by the disk among sets of constant width.
\end{proposition}

\begin{proof}
 The case of simple eigenvalues corresponds to $m=0$. In lemma \ref{Simple}, we have computed the simple eigenvalue of Dirichlet Laplacian on $\Omega_{\varepsilon}$:

 \begin{equation}\label{eq.R1}
  \lambda(\Omega_{\varepsilon})= j_{0,p}^{2}\Bigg(1+4\varepsilon^{2}\sum_{k}\Bigg(\frac{1}{2}+\frac{1}{2} k^{2}+j_{0,p}\frac{J_{k}^{'}(j_{0,p})}{J_{k}(j_{0,p})}\Bigg)|a_{l}|^{2}+O(\varepsilon^{3})\Bigg),
\end{equation}

 so if we can find $k$ such that
$$
\frac{1}{2}+\frac{k^{2}}{2}+j_{0,p}\frac{J_{k}^{'}(j_{0,p})}{J_{k}(j_{0,p})}< 0,\quad k \mbox{ is odd}
$$
we can show that the disk is not a local minimizer. In fact,

$$
\frac{1}{2}+\frac{k^{2}}{2}+j_{0,p}\frac{J_{k}^{'}(j_{0,p})}{J_{k}(j_{0,p})}< 0,\mbox{ for } k=3\mbox{ and }\forall p\geq 2.
$$

From $\jnp\frac{J_{m+3}^{'}(\jnp)}{J_{m+3}(\jnp)}=-(m+3)+\frac{2(m+1)\jnpp}{4(m+2)(m+1)-\jnpp}$(see appendix \ref{sec:Appendix}) and for $m=0$, we have
$
j_{0,p}\frac{J_{3}^{'}(j_{0,p})}{J_{3}(j_{0,p})}=-3+\frac{2j_{0,p}^{2}}{8-j_{0,p}^{2}}.
$
Then
$$
1+(3)^{2}+2j_{0,p}\frac{J_{3}^{'}(j_{0,p})}{J_{3}(j_{0,p})}=10-6+\frac{4j_{0,p}^{2}}{8-j_{0,p}^{2}}=\frac{32}{8-j_{0,p}^{2}}.
$$
As $j_{0,p}\in \;]2\sqrt{2},+\infty[$ $\forall p\geq 2$, $\;\frac{32}{8-j_{0,p}^{2}}< 0$. If we choose $a_{3}\neq 0$ and all other terms equal to zero, we can see that
$$
\lambda_{\kappa}(\Omega_{\varepsilon})=j_{0,p}^{2}\Bigg(1+4\varepsilon^{2}\Bigg(\underbrace{\frac{1}{2}+\frac{3^{2}}{2}+j_{0,p}\frac{J_{3}^{'}(j_{0,p})}{J_{3}^{'}(j_{0,p})}}_{<0}\Bigg)|a_{3}|^{2}+O(\varepsilon^{3})\Bigg)
$$
$$
\Rightarrow \lambda_{\kappa}(\Omega_{\varepsilon})\leq j_{0,p}^{2}=\lambda_{\kappa}(\Omega_{0})\quad\quad\forall p\geq 2.
$$
\end{proof}

\begin{proposition}\label{lemmedouble}
Let $\lambda_{\kappa}$ be an eigenvalue of the Dirichlet-Laplacian which is double for the disk. The disk is not a local minimizer among sets of constant width for $\lambda_{\kappa}$ where $\kappa\in \mathbb{N}^{*}\setminus\{2,3,4,5,7,8,11,12,16,$ $17,26,27,33,34,41,42,49,50\}$.
\end{proposition}

\begin{proof}
We are in the case of double eigenvalue $(m\neq 0)$ and by the lemma \ref{Double}, $\lambda(\Omega_{\varepsilon})$ is given by

$$
\lambda (\Omega_{\varepsilon}) =j_{m,p}^{2}\Bigg[1+2\varepsilon^{2}\Bigg(\sum_{|l|\neq m}\Bigg(\frac{1}{2}+\frac{1}{2}(m-l)^{2}+j_{m,p}\frac{J_{l}^{'}(j_{m,p})}{J_{l}(j_{m,p})}\Bigg)|a_{m-l}|^{2}
$$
$$
\quad\quad + \underbrace{\frac{\alpha_{m}}{\overline{\alpha_{m}}}
\Bigg(\sum_{l\neq m}\frac{1}{2}-\frac{1}{2}(m^{2}-l^{2})+j_{m,p}\frac{J_{l}^{'}(j_{m,p})}{J_{l}(j_{m,p})}\Bigg)a_{m+l}a_{m-l}\Bigg)}_{\Upsilon}+O(\varepsilon^{3})\Bigg].
$$
Taken $k=m-l$, we have
$$
\lambda (\Omega_{\varepsilon}) =j_{m,p}^{2}\Bigg[1+2\varepsilon^{2}\Bigg(\displaystyle\sum\limits_{ k\in \mathbb{N},k\neq0,k\neq2m}\Bigg(\underbrace{1+k^{2}+\jnp\frac{J_{k+m}^{'}(\jnp)}{J_{k+m}(\jnp)}+\jnp\frac{J_{k-m}^{'}(\jnp)}{J_{k-m}(\jnp)}}_{C_{k,m}(j_{m,p})}\Bigg)|a_{k}|^{2}
$$
$$
\quad\quad + \underbrace{\frac{\alpha_{m}}{\overline{\alpha_{m}}}
\Bigg(\sum_{l\neq m}\frac{1}{2}-\frac{1}{2}(m^{2}-l^{2})+j_{m,p}\frac{J_{l}^{'}(j_{m,p})}{J_{l}(j_{m,p})}\Bigg)a_{m+l}a_{m-l}\Bigg)}_{\Upsilon}+O(\varepsilon^{3})\Bigg].
$$

If we can find $k$ such that

\begin{equation}\label{zak.5}
C_{k,m}(j_{m,p})= 1+k^{2}+j_{m,p}\frac{J_{k-m}^{'}(j_{m,p})}{J_{k-m}(j_{m,p})}+j_{m,p}\frac{J_{k+m}^{'}(j_{m,p})}{J_{k+m}(j_{m,p})}< 0
 \end{equation}
we can not conclude if the disk is a local minimizer or not, because we don't know the sign of the second term $\Upsilon$. For this reason we have to look for  particular $k$ which verifies (\ref{zak.5}) and makes $\Upsilon$ vanish. In this case we will be able to find $a_{n}$ and $b_{n}$ satisfying the necessary conditions and such that $\lambda(\Omega_{\varepsilon})\leq \lambda(\Omega_{0})$, which proves the above proposition.\\

Firstly we prove the statement for $m\geq 9$, then we prove the remaining cases one by one. In this proof we need some inequality. In order to make clear the proof  we will prove them in Appendix \ref{sec:AppendixB}.
\begin{description}
  \item[$\bullet$] For $m\geq 9$ and $\forall p\in \mathbb{N}^{*}$,  we have that
\begin{equation}\label{C3mC5m}
  C_{3,m}(j_{m,p})< 0 \quad \mbox{if}\; j_{m,p}\in I
\end{equation}

where $I=[\sqrt{m(m+2)}, 2\sqrt{(m-1)(m-2)}[\cup [ 2\sqrt{(m+1)(m+2)},+\infty[$,\\

and

\begin{equation}\label{C3mC5m}
  C_{5,m}(j_{m,p})< 0 \quad \mbox{if}\; j_{m,p}\in V
\end{equation}
where $V=[ 2\sqrt{(m-1)(m-2)},2\sqrt{(m+1)(m+2)}[$.\\

It should be noted that we do not take into consideration $j_{m,p}\in [0,\sqrt{m(m+2)}[$ because there is no $j_{m,p}$ in this interval and it comes from the inequality given by G.Watson (\cite{wat1995},p.486), $j_{m,1}\geq \sqrt{m(m+2)}$ which implies $j_{m,p}\geq \sqrt{m(m+2)}\; \forall p\in \mathbb{N}^{*}$. We are going to prove that the disk is not a local minimizer for the eigenvalue corresponding to $j_{m,p}$ with $m\geq 9$ and $\forall p\in \mathbb{N}^{*}$. \\

If $j_{m,p}\in I$, we can take $a_{3}\neq 0$ and $a_{j}=0\quad \forall j\neq 3$, so $a_{m-l}a_{m+l}=0\; \forall l\in \mathbb{Z}$ and $\Upsilon=0$. Then
$$
\lambda (\Omega_{\varepsilon})=j_{m,p}^{2}\Bigg[1+2\varepsilon^{2}\Bigg(\underbrace{C_{3,m}(j_{m,p}}_{< 0})\Bigg)|a_{3}|^{2}+O(\varepsilon)^{3}\Bigg]\leq j_{m,p}^{2}
$$
$$
\Rightarrow \lambda (\Omega_{\varepsilon})\leq \lambda (\Omega_{0}).
$$

If $j_{m,p}\in V$, we can take $a_{5}\neq 0$ and $a_{j}=0\quad \forall j\neq 5$, then $\Upsilon=0$ and
$$
\lambda (\Omega_{\varepsilon})=j_{m,p}^{2}\Bigg[1+2\varepsilon^{2}\Bigg(\underbrace{C_{5,m}(j_{m,p}}_{< 0})\Bigg)|a_{5}|^{2}+O(\varepsilon)^{3}\Bigg]\leq j_{m,p}^{2}
$$
$$
\Rightarrow \lambda (\Omega_{\varepsilon})\leq \lambda (\Omega_{0}).
$$

  \item[$\bullet$] For $m=1$, $m=2$, we have that
\begin{equation}\label{C31}
  C_{3,1}(j_{1,p})< 0\quad \forall p\geq 2
\end{equation}
\begin{equation}\label{C32}
  C_{3,2}(j_{2,p})< 0\quad \forall p\geq 2
\end{equation}

It is sufficient to choose $a_{3}\neq0$ and $a_{j}= 0\quad \forall j\neq 3$, to see that for $m=1$ or $m=2$ we have
$
a_{m-l}a_{m+l}=0 \; \forall l\in \mathbb{Z}.
$,
By (\ref{C31}) and (\ref{C32}), one has
$$
\lambda (\Omega_{\varepsilon})=j_{m,p}^{2}\Bigg[1+2\varepsilon^{2}\Bigg(\underbrace{C_{3,m}(j_{m,p}}_{<0})\Bigg)|a_{3}|^{2}+O(\varepsilon^{3})\Bigg]\leq j_{m,p}^{2}
$$
and therefore
$$
\lambda (\Omega_{\varepsilon})\leq \lambda (\Omega_{0})
$$

  \item[$\bullet$] For $m=3$

  \begin{equation}\label{C33}
  C_{3,3}(j_{3,p})< 0\quad\quad \forall p\geq 2
\end{equation}
\begin{equation}\label{App41}
  C_{5,3}(j_{3,p})< 0\quad\quad \forall p\geq 5
\end{equation}
We are going to prove that the disk is not a local minimizer for the eigenvalue corresponding to $j_{m,p}$ with $m=3$ and $\forall p\geq 2$. \\

In this case, by choosing $a_{3}\neq0$ and $a_{j}=0\; \forall j\neq 3$, the second term in the expression of $\lambda(\Omega_{\varepsilon})$ in lemma \ref{Double} doesn't vanish and we can not conclude. For this reason, we choose $k=5$. \\

By using (\ref{App41}), we have
$$
\lambda (\Omega_{\varepsilon})=j_{3,p}^{2}\Bigg[1+2\varepsilon^{2}\underbrace{C_{5,3}(j_{3,p})}_{<0}|a_{5}|^{2}+O(\varepsilon^{3})\Bigg]\leq j_{3,p}^{2}\quad \forall p\geq 5
$$
and therefore
$$
 \lambda (\Omega_{\varepsilon})\leq \lambda (\Omega_{0}).
$$

We still have three terms $j_{3,2}$, $j_{3,3}$ and $j_{3,4}$. By using lemma \ref{Double}

\begin{equation*}
\begin{array}{lll}
\lambda (\Omega_{\varepsilon})&=&j_{3,p}^{2}\Bigg[1+2\varepsilon^{2}\Bigg(\sum\limits_{|l|\neq 3}\Bigg(\frac{1}{2}+\frac{1}{2}(3-l)^{2}+j_{3,p}\frac{J_{l}^{'}(j_{3,p})}{J_{l}(j_{3,p})}\Bigg)|a_{3-l}|^{2}
\\
&&+\frac{\alpha_{3}}{\overline{\alpha_{3}}}\Bigg(\sum\limits_{l\neq 3}\frac{1}{2}-\frac{1}{2}(3^{2}-l^{2})+j_{3,p}\frac{J_{l}^{'}(j_{3,p})}{J_{l}(j_{3,p})}\Bigg)a_{3+l}a_{3-l}\Bigg)+O(\varepsilon^{3})\Bigg]
\\
&=&j_{3,p}^{2}\Bigg[1+2\varepsilon^{2}\Bigg(\sum\limits_{|l|\neq 3}\Bigg(\frac{1}{2}+\frac{1}{2}(3-l)^{2}+j_{3,p}\frac{J_{l}^{'}(j_{3,p})}{J_{l}(j_{3,p})}\Bigg)|a_{3-l}|^{2}
\\
&&\pm\; \Bigg|\Bigg(\sum\limits_{l\neq 3}\frac{1}{2}-\frac{1}{2}(3^{2}-l^{2})+j_{3,p}\frac{J_{l}^{'}(j_{3,p})}{J_{l}(j_{3,p})}\Bigg)a_{3+l}a_{3-l}\Bigg)\Bigg|+O(\varepsilon^{3})\Bigg].
\end{array}
\end{equation*}
For $k=3-l$,
\begin{eqnarray*}
\lambda (\Omega_{\varepsilon})   &=&j_{3,p}^{2}\Bigg[1+2\varepsilon^{2}\Bigg(\displaystyle\sum\limits_{ k\in \mathbb{N},k\neq0,k\neq6}\Bigg(1+k^{2}+j_{3,p}\frac{J_{k-3}^{'}(j_{3,p})}{J_{k-3}(j_{3,p})}(j_{3,p})\\
&&\quad \quad +j_{3,p}\frac{J_{k+3}^{'}(j_{3,p})}{J_{k+3}(j_{3,p})}\Bigg)|a_{k}|^{2}
\\
&&\pm\; \Bigg|\Bigg(\sum\limits_{l\neq 3}\frac{1}{2}-\frac{1}{2}(3^{2}-l^{2})+j_{3,p}\frac{J_{l}^{'}(j_{3,p})}{J_{l}(j_{3,p})}\Bigg)a_{3+l}a_{3-l}\Bigg)\Bigg|+O(\varepsilon^{3})\Bigg].
\end{eqnarray*}

For $a_{3}\neq 0$ and $a_{j}= 0$ $\forall j\not= 3$, we have

\begin{equation}
\begin{array}{lll}
\lambda (\Omega_{\varepsilon})&=&j_{3,p}^{2}\Bigg[1+2\varepsilon^{2}\Bigg((10+j_{3,p}\frac{J_{0}^{'}(j_{3,p})}{J_{0}(j_{3,p})}+j_{3,p}\frac{J_{6}^{'}(j_{3,p})}{J_{6}(j_{3,p})})|a_{3}|^{2}
\\
                      &&\pm\quad \Bigg|-4+j_{3,p}\frac{J_{0}^{'}(j_{3,p})}{J_{0}(j_{3,p})}\Bigg|(a_{3})^{2}\Bigg)\Bigg]+O(\varepsilon^{3}).
\end{array}
\end{equation}

In order to have $(a_{3})^{2}=|a_{3}|^{2}$, we can take $a_{3}\in \mathbb{R}$,

\begin{equation}\label{Ja.4}
\begin{array}{lll}
\lambda (\Omega_{\varepsilon})&=&j_{3,p}^{2}\Bigg[1+2\varepsilon^{2}\Bigg(10+j_{3,p}\frac{J_{0}^{'}(j_{3,p})}{J_{0}(j_{3,p})}+j_{3,p}\frac{J_{6}^{'}(j_{3,p})}{J_{6}(j_{3,p})}
\\
                      &&\pm\quad \Bigg|-4+j_{3,p}\frac{J_{0}^{'}(j_{3,p})}{J_{0}(j_{3,p})}\Bigg||a_{3}|^{2}\Bigg)\Bigg]+O(\varepsilon^{3}).
\end{array}
\end{equation}

By using $\jnp\frac{J_{m-3}^{'}(\jnp)}{J_{m-3}(\jnp)}=(m-3)-\frac{2(m-1)\jnpp}{4(m-2)(m-1)-\jnpp}$ (see Appendix \ref{sec:Appendix}), for $m=3$, we have
\begin{eqnarray*}
 -4+j_{3,p}\frac{J_{0}^{'}(j_{3,p})}{J_{0}(j_{3,p})}  &=& -4+\frac{-4j_{3,p}^{2}}{8-j_{3,p}^{2}} \\
                                                      &=& \frac{-32}{8-j_{3,p}^{2}}>0\quad \forall p\geq 2.
\end{eqnarray*}
\begin{description}
  \item[$\bullet$] In equality(\ref{Ja.4}) with the sign $(+)$, one has\\
 $$
\Bigg(10+j_{3,p}\frac{J_{0}^{'}(j_{3,p})}{J_{0}(j_{3,p})}+j_{3,p}\frac{J_{6}^{'}(j_{3,p})}{J_{p}(j_{3,p})}+\Bigg|-4+j_{3,p}\frac{J_{0}^{'}(j_{3,p})}{J_{0}(j_{3,p})}\Bigg|\Bigg).
$$
$$
=6+2j_{3,p}\frac{J_{0}^{'}(j_{3,p})}{J_{0}(j_{3,p})}+j_{3,p}\frac{J_{6}^{'}(j_{3,p})}{J_{p}(j_{3,p})}
$$
By using $\jnp\frac{J_{m+3}^{'}(\jnp)}{J_{m+3}(\jnp)}=-(m+3)+\frac{2(m+1)\jnpp}{4(m+2)(m+1)-\jnpp}$ (see Appendix \ref{sec:Appendix}), for $m=3$, one gets
$$
j_{3,p}\frac{J_{6}^{'}(j_{3,p})}{J_{p}(j_{3,p})}=-6+\frac{8j_{3,p}^{2}}{80-j_{3,p}^{2}}\mbox{and}\;j_{3,p}\frac{J_{0}^{'}(j_{3,p})}{J_{0}(j_{3,p})}=\frac{-4j_{3,p}^{2}}{8-j_{3,p}^{2}}
$$
Finally,
$$
\Bigg(10+j_{3,p}\frac{J_{0}^{'}(j_{3,p})}{J_{0}(j_{3,p})}+j_{3,p}\frac{J_{6}^{'}(j_{3,p})}{J_{p}(j_{3,p})}+\Bigg|-4+j_{3,p}\frac{J_{0}^{'}(j_{3,p})}{J_{0}(j_{3,p})}\Bigg|\Bigg)
$$
$$
=\frac{-576j_{3,p}^{2}}{(8-j_{3,p}^{2})(80-j_{3,p}^{2})}<0\;,\;j_{3,p}>\sqrt{80}\; \forall p\geq 2.
$$
  \item[$\bullet$] In equality (\ref{Ja.4}) with the sign $(-)$, one has\\
 $$
\Bigg(10+j_{3,p}\frac{J_{0}^{'}(j_{3,p})}{J_{0}(j_{3,p})}+j_{3,p}\frac{J_{6}^{'}(j_{3,p})}{J_{p}(j_{3,p})}-\Bigg|-4+j_{3,p}\frac{J_{0}^{'}(j_{3,p})}{J_{0}(j_{3,p})}\Bigg|\Bigg).
$$
$$
=14+j_{3,p}\frac{J_{6}^{'}(j_{3,p})}{J_{p}(j_{3,p})}=8+\frac{8j_{3,p}^{2}}{80-j_{3,p}^{2}}=\frac{640}{80-j_{3,p}^{2}}<0\;\forall p\geq 2.
$$
\end{description}

We conclude that the disk is not a local minimizer for the eigenvalues corresponding to $j_{3,2},j_{3,3}, j_{3,4}$.\\

\item[$\bullet$] For $m=4$, $m=5$, $m=6$, $m=7$, $m=8$:\\

The proof is the same, we just need to know $|a_{k}|$ such that $C_{k,m}(j_{m,p})<0$ and we take $a_{k}\neq 0$ and $a_{l}= 0$, for $l\neq k$. So,

$$
\lambda (\Omega_{\varepsilon})=j_{m,p}^{2}\Bigg[1+2\varepsilon^{2}\underbrace{C_{k,m}(j_{m,p})}_{\leq 0}|a_{k}|^{2}+O(\varepsilon)^{3}\Bigg]< j_{m,p}^{2}
$$
$$
\Rightarrow \lambda (\Omega_{\varepsilon})\leq \lambda (\Omega_{0})
$$

\begin{description}
    \item[$\bullet$] For $m= 4$
\begin{equation}\label{C34}
C_{3,4}(j_{4,p})< 0\quad\quad \forall p\geq 2
\end{equation}

  \item[$\bullet$] For $m= 5$
\begin{equation}\label{C35}
C_{3,5}(j_{5,p})< 0\quad\quad \forall p\geq 3
\end{equation}
  \item[$\bullet$] For $m= 6$
\begin{equation}\label{C56}
C_{5,6}(j_{6,1})< 0
\end{equation}
\begin{equation}\label{C36}
C_{3,6}(j_{6,p})< 0\quad\quad \forall p\geq 3
\end{equation}
  \item[$\bullet$] For $m= 7$
\begin{equation}\label{C37}
C_{3,7}(j_{7,p})< 0\quad\quad  \forall p\geq 3
\end{equation}
\item[$\bullet$] For $m=8$
\begin{equation}\label{C38}
C_{3,8}(j_{8,p})< 0\quad\quad \forall p\geq 3
\end{equation}

\begin{equation}\label{C38andC58}
C_{3,8}(j_{8,1})< 0\quad and \quad  C_{5,8}(j_{8,2})< 0
\end{equation}
\end{description}
\end{description}
\end{proof}

\begin{remark}
Concerning the zero of the bessel function $j_{7,p}$, for $p=1$  which corresponds to $\lambda_{49}$ and $\lambda_{50}$, we are ranking them as open cases for two reasons. On one hand, the answer to this question requires a lot of calculation which is going to make the paper much longer. On the other hand,  when we try to solve the question related to $\lambda_{49}$, we find that the disk is not a local minimizer while for $\lambda_{50}$, we are not able to decide whether the disk is a local minimizer or not.

We conjecture that the disk is not a local minimizer for the open cases and its a global minimizer for $\lambda_{3}$ among open set of constant width.

\end{remark}


\appendix
\section{Some ratios of Bessel functions}
\label{sec:Appendix}

We consider the following classical results on Bessel Functions:
\begin{equation}\label{eq.119}
  \forall n\in \mathbb{N},\quad \forall x\in \mathbb{R}_{+}^{*},\quad xJ_{n}^{'}=nJ_{n}-xJ_{n+1}
\end{equation}
\begin{equation}\label{eq.120}
  \forall n\in \mathbb{N},\quad \forall x\in \mathbb{R}_{+}^{*},\quad xJ_{n}^{'}=-nJ_{n}+xJ_{n-1}
\end{equation}
\begin{equation}\label{eq.121}
  \forall n\in \mathbb{N},\quad \forall x\in \mathbb{R}_{+}^{*},\quad \frac{2n}{x}J_{n}=J_{n-1}+J_{n+1}
\end{equation}

These result can be found in \cite{wat1995}, p45.\\

$\bullet$ Computation of $\displaystyle j_{m,p}\frac{J_{m-1}^{'}(j_{m,p})}{J_{m-1}(j_{m,p})}$\\

From (\ref{eq.119}) and for $n=m-1$, we deduce that
\begin{equation}\label{Jmm1}
j_{m,p}\frac{J_{m-1}^{'}(j_{m,p})}{J_{m-1}(j_{m,p})}=(m-1).
\end{equation}

$\bullet$ Computation of $\displaystyle j_{m,p}\frac{J_{m+1}^{'}(j_{m,p})}{J_{m+1}(j_{m,p})}$\\

From (\ref{eq.120}) and for $n=m+1$, we conclude that
\begin{equation}\label{Jmp1}
j_{m,p}\frac{J_{m+1}^{'}(j_{m,p})}{J_{m+1}(j_{m,p})}=-(m+1).
\end{equation}

$\bullet$ Computation of $\displaystyle j_{m,p} \frac{J_{m-3}^{'}(j_{m,p})}{J_{m-3}(j_{m,p})}$:\\

From (\ref{eq.121}), we have that
$
\frac{2(m-1)}{\jnp}J_{m-1}(\jnp)=J_{m-2}(\jnp)+J_{m}(\jnp)
$

and therefore
\begin{equation}\label{Ap1}
\frac{J_{m-1}(\jnp)}{J_{m-2}(\jnp)}=\frac{\jnp}{2(m-1)}.
\end{equation}
Again by (\ref{eq.121}):
$\frac{2(m-2)}{j_{m,p}}J_{m-2}(j_{m,p})=J_{m-3}(j_{m,p})+J_{m-1}(j_{m,p})
$,
that implies
$
 \frac{J_{m-3}(j_{m,p})}{J_{m-2}(j_{m,p})}=\frac{2(m-2)}{j_{m,p}}-\frac{J_{m-1}(j_{m,p})}{J_{m-2}(j_{m,p})}
$,
and by  (\ref{Ap1}) we get
\begin{equation}\label{eq.X}
\frac{J_{m-3}(\jnp)}{J_{m-2}(\jnp)}= \frac{2(m-2)}{\jnp}-\frac{\jnp}{2(m-1)} = \frac{4(m-1)(m-2)-\jnpp}{2\jnp (m-1)}.
\end{equation}
From (\ref{eq.119}), we have
$
j_{m,p} J_{m-3}^{'}(j_{m,p})=(m-3)J_{m-3}(j_{m,p})-j_{m,p} J_{m-2}(j_{m,p})
$
and therefore
$$
 j_{m,p} \frac{J_{m-3}^{'}(j_{m,p})}{J_{m-3}(j_{m,p})}=m-3-j_{m,p} \frac{J_{m-2}(j_{m,p})}{J_{m-3}(j_{m,p})},
$$
and by (\ref{eq.X}) we obtain
\begin{equation}\label{eq.AAA1}
\jnp\frac{J_{m-3}^{'}(\jnp)}{J_{m-3}(\jnp)}=(m-3)-\frac{2(m-1)\jnpp}{4(m-2)(m-1)-\jnpp}.
\end{equation}

$\bullet$ Computation of $\displaystyle j_{m,p} \frac{J_{m+3}^{'}(j_{m,p})}{J_{m+3}(j_{m,p})}$:\\

From (\ref{eq.121}) and $J_{m}(j_{m,p})=0$ we have that
$
\frac{2(m+1)}{\jnp}J_{m+1}(\jnp)=J_{m}(\jnp)+J_{m+2}(\jnp)
$
that implies
\begin{equation}\label{Ap.2}
 \frac{J_{m+1}(\jnp)}{J_{m+2}(\jnp)}=\frac{\jnp}{2(m+1)}.
\end{equation}
By using (\ref{eq.121}) we have
$
\frac{2(m+2)}{j_{m,p}}J_{m+2}(j_{m,p})=J_{m+1}(j_{m,p})+J_{m+3}(j_{m,p})
$,
that implies
$
 \frac{J_{m+3}(j_{m,p})}{J_{m+2}(j_{m,p})}=\frac{2(m+2)}{j_{m,p}}-\frac{J_{m+1}(j_{m,p})}{J_{m+2}(j_{m,p})}
$. By (\ref{Ap.2})
\begin{equation}\label{A}
\frac{J_{m+3}(\jnp)}{J_{m+2}(\jnp)}= \frac{2(m+2)}{\jnp}-\frac{\jnp}{2(m+1)} = \frac{4(m+1)(m+2)-\jnpp}{2\jnp (m+1)}.
\end{equation}
From (\ref{eq.120}) we have that
$
j_{m,p} J_{m+3}^{'}(j_{m,p})=-(m+3)J_{m+3}(j_{m,p})+j_{m,p} J_{m+2}(j_{m,p})
$
and therefore
$
 j_{m,p} \frac{J_{m+3}^{'}(j_{m,p})}{J_{m+3}(j_{m,p})}=-(m+3)+j_{m,p} \frac{J_{m+2}(j_{m,p})}{J_{m+3}(j_{m,p})}.
$
Finally, using (\ref{A}) we get
\begin{equation}\label{Jmp3}
\jnp\frac{J_{m+3}^{'}(\jnp)}{J_{m+3}(\jnp)}=-(m+3)+\frac{2(m+1)\jnpp}{4(m+2)(m+1)-\jnpp}
\end{equation}

$\bullet$ Computing of $\displaystyle j_{m,p} \frac{J_{m+5}^{'}(j_{m,p})}{J_{m+5}(j_{m,p})}$:\\

From (\ref{eq.121}), we have that $\;
\frac{J_{m+4}(j_{m,p})}{J_{m+3}(j_{m,p})}=\frac{2(m+3)}{j_{m,p}}-\frac{J_{m+2}(j_{m,p})}{J_{m+3}(j_{m,p})}.
$
By using equality (\ref{A}), we deduce
\begin{equation}\label{eq.B}
\begin{array}{lll}
\frac{J_{m+4}(\jnp)}{J_{m+3}(\jnp)} &=&\frac{2(m+3)}{\jnp}-\frac{2\jnp (m+1)}{4(m+1)(m+2)-\jnpp}
\\
&=& \frac{8(m+3)(m+2)(m+1)-4\jnpp (m+2)}{\jnp(4(m+1)(m+2)-\jnpp)}.
\end{array}
\end{equation}

Therefore, (\ref{eq.121}) gives that
$
\frac{J_{m+5}(\jnp)}{J_{m+4}(\jnp)}=\frac{2(m+4)}{\jnp}-\frac{J_{m+3}(\jnp)}{J_{m+4}(\jnp)}
$. Using (\ref{eq.B}), we have
$$
\frac{J_{m+5}(\jnp)}{J_{m+4}(\jnp)}=\frac{2(m+4)}{\jnp}-\frac{\jnp (4(m+2)(m+1)-\jnpp)}{8(m+3)(m+2)(m+1)-4\jnpp (m+2)}
$$
\begin{equation}\label{eq.C}
\quad\quad= \frac{16(m+4)(m+3)(m+2)(m+1)-4\jnpp (m+2)(3m+9)+j_{m,p}^{4}}{\jnp (8(m+3)(m+2)(m+1)-4\jnpp (m+2))}.
\end{equation}
From (\ref{eq.120}), we have that
$
\jnp\frac{J_{m+5}^{'}(\jnp)}{J_{m+5}(\jnp)}=-(m+5)+\jnp\frac{J_{m+4}(\jnp)}{J_{m+5}(\jnp)}
$. Using (\ref{eq.C}), we deduce that
\begin{equation}\label{E}
\begin{array}{lll}
\jnp\frac{J_{m+5}^{'}(\jnp)}{J_{m+5}(\jnp)}&=&-(m+5)
\\
&&+\frac{\jnpp \Bigg(8(m+3)(m+2)(m+1)-4\jnpp (m+2)\Bigg)}{16(m+4)(m+3)(m+2)(m+1)-4\jnpp (m+2)(3m+9)+j_{m,p}^{4}}
\end{array}
\end{equation}

$\bullet$ Computation of $\displaystyle j_{m,p} \frac{J_{m-5}^{'}(j_{m,p})}{J_{m-5}(j_{m,p})}$:\\

From (\ref{eq.121}), we have that
$
\;\frac{J_{m-4}(\jnp)}{J_{m-3}(\jnp)}=\frac{2(m-3)}{\jnp}-\frac{J_{m-2}(\jnp)}{J_{m-3}(\jnp)}
$. By using equality (\ref{eq.X}):
\begin{equation}\label{BB}
\begin{array}{lll}
\frac{J_{m-4}(\jnp)}{J_{m-3}(\jnp)} &=& \frac{2(m-3)}{\jnp}-\frac{2\jnp (m-1)}{4(m-1)(m-2)-\jnpp}
\\
&=& \frac{8(m-3)(m-2)(m-1)-4\jnpp (m-2)}{\jnp(4(m-1)(m-2)-\jnpp)}.
\end{array}
\end{equation}
Therefore, (\ref{eq.121}) gives that
$
\frac{J_{m-5}(\jnp)}{J_{m-4}(\jnp)}=\frac{2(m-4)}{\jnp}-\frac{J_{m-3}(\jnp)}{J_{m-4}(\jnp)}
$, and using (\ref{BB}), we have
$$
\frac{J_{m-5}(\jnp)}{J_{m-4}(\jnp)}=\frac{2(m-4)}{\jnp}-\frac{\jnp (4(m-2)(m-1)-\jnpp)}{8(m-3)(m-2)(m-1)-4\jnpp (m-2)}
$$
$$
\quad\quad=\frac{16(m-4)(m-3)(m-2)(m-1)-4\jnpp (m-2)(3m-9)+j_{m,p}^{4}}{\jnp (8(m-3)(m-2)(m-1)-4\jnpp (m-2))}
$$

From (\ref{eq.119}), we have that
$
\jnp\frac{J_{m-5}^{'}(\jnp)}{J_{m-5}(\jnp)}=(m-5)-\jnp\frac{J_{m-4}(\jnp)}{J_{m-5}(\jnp)}
$. Finally, we get

\begin{equation}\label{D}
\begin{array}{lll}
\jnp\frac{J_{m-5}^{'}(\jnp)}{J_{m-5}(\jnp)}=(m-5)
\\
\qquad\qquad-\frac{\jnpp \Bigg(8(m-3)(m-2)(m-1)-4\jnpp (m-2)\Bigg)}{16(m-4)(m-3)(m-2)(m-1)-4\jnpp (m-2)(3m-9)+j_{m,p}^{4}}
\end{array}
\end{equation}

\section{Proof of lemma \ref{P}}
\label{sec:AppendixB2}

\textbf{Lemma 6}

  \begin{equation}
 C_{k,m}(j_{m,p})=1+k^{2}+\jnp\frac{J_{k+m}^{'}(\jnp)}{J_{k+m}(\jnp)}+\jnp\frac{J_{k-m}^{'}(\jnp)}{J_{k-m}(\jnp)}\geq 0
\end{equation}
for all $k$ odd natural number and $\jnp\in\Bigg\{j_{1,1}; j_{2,1}; j_{3,1}; j_{4,1}; j_{5,1}; j_{5,2}; j_{6,2}; j_{7,1}\Bigg\}$.
\begin{proof}
To prove this lemma, we use the ratios of Bessel functions given in the previous Appendix and the graph of the function $x\frac{J_{n}^{'}(x)}{J_{n}(x)}$ given in \cite{land1999}. We give the proof for the small values $j_{1,1}$, $j_{2,1}$, $j_{3,1}$ and $j_{4,1}$ and the large one $j_{6,2}$. The other cases are obtained by the same reasoning. We set
\begin{equation}\label{Ckmx}
C_{k,m}(x)=1+k^{2}+x\frac{J_{k+m}^{'}(x)}{J_{k+m}(x)}+x\frac{J_{k-m}^{'}(x)}{J_{k-m}(x)}
\end{equation}
 and
$$
F_{n}(x)=x\frac{J_{n}^{'}(x)}{J_{n}(x)},\quad x>0\quad,\quad n\in \mathbb{N}.
$$
J. Landau in (\cite{land1999}, p 194) gives a detailed picture of the graph of $F_{n}(x)$. $F_{n}(x)$ decreases from $n$ at $x=0$ to $-\infty$ at $x=j_{n,1}$, jumping to $+\infty$  and decreases to $-\infty$ in each interval $]j_{m,p},j_{m,p+1}[$ for all natural number $p\geq 1$.

\begin{equation}\label{J.L1}
  \mbox{If}\; 0<x\leq j_{n,1}^{'}\; \mbox{this implies}\;F_{n}(x)=x\frac{J_{n}^{'}(x)}{J_{n}(x)}\geq 0.
\end{equation}

It is clear that $j_{k-m,1}^{'}\leq j_{k+m,1}^{'}$. Then, for all real positif $x$, such that $x\leq j_{k-m,1}^{'}$, we have: $F_{k-m}(x)\geq0, \mbox{ and } F_{k+m}(x)\geq 0$. This implies

\begin{equation}\label{JL}
  C_{k,m}(x)\geq 0\mbox{ for all } 0<x\leq j_{k-m,1}^{'}.
\end{equation}

\begin{description}
  \item[$\bullet$] For $j_{1,1}$:\\
We are looking for the values of $k$ which verify
$$
j_{1,1}\leq j_{k-1,1}^{'}.
$$
We have $j_{1,1}\leq j_{3,1}^{'}$ (see table 1 and table 2) in Appendix D. \\
 $$j_{1,1}\leq j_{k-1,1}^{'}\Rightarrow k-1\geq 3\Rightarrow k\geq 4$$.\\
  So by (\ref{JL}), $C_{k,1}(j_{1,1})\geq 0 \mbox{ for } k\geq 4$.\\
On the other hand, for $k=1$, we have
$$
C_{1,1}(j_{1,1})=1+1+j_{1,1}\frac{J_{2}^{'}(j_{1,1})}{J_{2}(j_{1,1})}+j_{1,1}\frac{J_{0}^{'}(j_{1,1})}{J_{0}(j_{1,1})}.$$
Using (\ref{Jmm1}) and (\ref{Jmp1}) for $m=1$ and $p=1$, we deduce that $C_{1,1}(j_{1,1})=0$.
Also, for $k=3$, we have: $$
C_{3,1}(j_{1,1})=1+9+j_{1,1}\frac{J_{4}^{'}(j_{1,1})}{J_{4}(j_{1,1})}+j_{1,1}\frac{J_{2}^{'}(j_{1,1})}{J_{2}(j_{1,1})}
.$$
By using (\ref{Jmp1}) and (\ref{Jmp3}) for $m=1$ and $p=1$, we have that $C_{3,1}(j_{1,1})=4+\frac{4j_{1,1}^{2}}{24-j_{1,1}^{2}}=\frac{96}{24-j_{1,1}^{2}}>0$, because $j_{1,1}< 2\sqrt{6}$. Finally, $C_{k,1}(j_{1,1})> 0$ for all $k$ odd natural number.
  \item[$\bullet$] For $j_{2,1}$:\\
We are looking for the values of $k$, which verify $
j_{2,1}\leq j_{k-2,1}^{'}.$
We have $j_{2,1}\leq j_{4,1}^{'}$ (see table 1 and table 2)in Appendix D.\\
  $$j_{2,1}\leq j_{k-2,1}^{'}\Rightarrow k-2\geq 4\Rightarrow k\geq 6$$
   So by (\ref{JL}) $C_{k,2}(j_{2,1})\geq 0 \mbox{ for } k\geq 6$.
   \begin{itemize}
     \item For $k=1$
     $$
C_{1,2}(j_{2,1})=2+j_{2,1}\frac{J_{3}^{'}(j_{2,1})}{J_{3}(j_{2,1})}+j_{2,1}\frac{J_{1}^{'}(j_{2,1})}{J_{1}(j_{2,1})},
$$
and by using (\ref{Jmm1}) and (\ref{Jmp1})  for $m=2$ and $p=1$, we obtain $C_{1,2}(j_{2,1})=0$.

     \item For $k=3$
$$
C_{3,2}(j_{2,1})=10+j_{2,1}\frac{J_{5}^{'}(j_{1,1})}{J_{5}(j_{2,1})}+j_{2,1}\frac{J_{1}^{'}(j_{2,1})}{J_{1}(j_{2,1})}.
$$
By using (\ref{Jmp1}) and (\ref{Jmp3}) for $m=2$ and $p=1$,  since $j_{2,1}< 4\sqrt{3}$, we have $C_{3,2}(j_{2,1})=\frac{288}{48-j_{2,1}^{2}}>0$.

     \item For $k=5$,
$$
C_{5,2}(j_{2,1})=26+j_{2,1}\frac{J_{7}^{'}(j_{2,1})}{J_{7}(j_{2,1})}+j_{2,1}\frac{J_{3}^{'}(j_{2,1})}{J_{3}(j_{2,1})}.
$$
By using (\ref{Jmp1}) and (\ref{E}) for $m=2$ and $p=1$, we have
$$
C_{5,2}(j_{2,1})=16+\frac{j_{2,1}^{2}(480-16j_{2,1}^{2})}{5760-240 j_{2,1}^{2}+j_{2,1}^{4}}
$$
$$
C_{5,2}(j_{2,1})=\frac{92160-3360j_{2,1}^{2}}{(j_{2,1}^{2}-(120+24\sqrt{15}))(j_{2,1}^{2}-(120-24\sqrt{15}))} \geq0,
$$
because $j_{2,1}\leq \sqrt{120-24\sqrt{15}} \leq \sqrt{\frac{92160}{3360}}\leq \sqrt{120+24\sqrt{15}}$. Therefore $C_{k,2}(j_{2,1})> 0$ for all odd natural numbers $k$.
   \end{itemize}
 \item[$\bullet$] For $j_{3,1}$:\\
 We have $j_{3,1}\leq j_{5,1}^{'}$ (see table 1 and table 2)in Appendix D.
  $$j_{3,1}\leq j_{k-3,1}^{'}\Rightarrow k-3\geq 5\Rightarrow k\geq 8$$
  So by (\ref{JL}), $C_{k,3}(j_{3,1})\geq 0 \mbox{ for } k\geq 8$.
 \begin{itemize}
   \item For $k=1$,

      $$C_{1,3}(j_{3,1})=1+1^2+j_{3,1}\frac{J_{2}^{'}(j_{3,1})}{J_{2}(j_{3,1})}+j_{3,1}\frac{J_{4}^{'}(j_{1,1})}{J_{4}(j_{3,1})}.$$

   By using (\ref{Jmm1}) and (\ref{Jmp1}), we get that $C_{1,3}(j_{3,1})=0$.

   \item For $k=3$

   $$C_{3,3}(j_{3,1})=1+3^2+j_{3,1}\frac{J_{0}^{'}(j_{3,1})}{J_{0}(j_{3,1})}+j_{3,1}\frac{J_{6}^{'}(j_{3,1})}{J_{6}(j_{3,1})}.$$
   We have $j_{3,1}\leq j_{6,1}^{'}$. By (\ref{J.L1}) this implies $j_{3,1}\frac{J_{6}^{'}(j_{3,1})}{J_{6}(j_{3,1})}\geq 0$. By using (\ref{eq.AAA1}) one gets:
   $$1+3^2+j_{3,1}\frac{J_{0}^{'}(j_{3,1})}{J_{0}(j_{3,1})}=10-\frac{4j^{2}_{3,1}}{8-j^{2}_{3,1}}=\frac{80-14j^{2}_{3,1}}{8-j^{2}_{3,1}} \geq 0, $$
   because $j_{3,1}\geq\sqrt{8} \geq\sqrt{\frac{80}{14}}$.

   \item For $k=5$,
   $$C_{5,3}(j_{3,1})=1+5^2+j_{3,1}\frac{J_{2}^{'}(j_{3,1})}{J_{2}(j_{3,1})}+j_{3,1}\frac{J_{8}^{'}(j_{3,1})}{J_{8}(j_{3,1})}.$$
   We have $j_{3,1}\leq j_{8,1}^{'}.$ By (\ref{J.L1}), this implies $j_{3,1}\frac{J_{8}^{'}(j_{3,1})}{J_{8}(j_{3,1})}\geq 0$.\\ By using (\ref{Jmm1}), one gets $\;1+5^2+j_{3,1}\frac{J_{2}^{'}(j_{3,1})}{J_{2}(j_{3,1})}=28 \geq 0$.

   \item For $k=7$
    $$C_{7,3}(j_{3,1})=1+7^2+j_{3,1}\frac{J_{4}^{'}(j_{3,1})}{J_{4}(j_{3,1})}+j_{3,1}\frac{J_{10}^{'}(j_{3,1})}{J_{10}(j_{3,1})}.$$
     We have $j_{3,1}\leq j_{10,1}^{'}.$ By (\ref{J.L1}), this implies $j_{3,1}\frac{J_{8}^{'}(j_{3,1})}{J_{8}(j_{3,1})}\geq 0$.\\
      By using (\ref{Jmp1}), we deduce $\;1+7^2+j_{3,1}\frac{J_{4}^{'}(j_{3,1})}{J_{4}(j_{3,1})}=46\geq 0$.
 \end{itemize}

  \item[$\bullet$] For $j_{4,1}$:\\
  We have $j_{4,1}\leq j_{7,1}^{'}$ (see table 1 and table 2) in Appendix D.
   $$j_{4,1}\leq j_{k-4,1}^{'}\Rightarrow k-4\geq 7\Rightarrow k\geq 11.$$
    So by (\ref{JL}), $\;C_{k,4}(j_{4,1})\geq 0 \mbox{ for } k\geq 11$.
 \begin{itemize}
   \item For $k=1$
   $$C_{1,4}(j_{4,1})=1+1^2+j_{4,1}\frac{J_{3}^{'}(j_{4,1})}{J_{3}(j_{4,1})}+j_{4,1}\frac{J_{5}^{'}(j_{4,1})}{J_{5}(j_{4,1})}.$$
     By using (\ref{Jmm1}) and (\ref{Jmp1}), we have that $C_{1,4}(j_{4,1})=0$.
   \item For $k=3$
      $$C_{3,4}(j_{4,1})=1+3^2+j_{4,1}\frac{J_{1}^{'}(j_{4,1})}{J_{1}(j_{4,1})}+j_{4,1}\frac{J_{7}^{'}(j_{4,1})}{J_{7}(j_{4,1})}.$$
     We have $j_{4,1}\leq j_{7,1}^{'}$. By (\ref{J.L1}) this implies $j_{4,1}\frac{J_{7}^{'}(j_{4,1})}{J_{7}(j_{4,1})}\geq 0$.\\
      By using (\ref{eq.AAA1}), one gets $$1+3^2+j_{4,1}\frac{J_{1}^{'}(j_{4,1})}{J_{1}(j_{4,1})}=\frac{264-17j_{4,1}^{2}}{24-j_{4,1}^{2}}\geq 0,$$
       because $j_{4,1}\geq\sqrt{24} \geq\sqrt{\frac{264}{17}}$.
   \item For $k=5$
        $$C_{5,4}(j_{4,1})=1+5^2+j_{4,1}\frac{J_{1}^{'}(j_{4,1})}{J_{1}(j_{4,1})}+j_{4,1}\frac{J_{9}^{'}(j_{4,1})}{J_{9}(j_{4,1})}.$$
     We have $j_{4,1}\leq j_{9,1}^{'}$. By (\ref{J.L1}), this implies $j_{4,1}\frac{J_{9}^{'}(j_{4,1})}{J_{9}(j_{4,1})}\geq 0$.\\
      By using (\ref{eq.AAA1}), one deduce $$1+5^2+j_{4,1}\frac{J_{1}^{'}(j_{4,1})}{J_{1}(j_{4,1})}=\frac{648-33j_{4,1}^{2}}{24-j_{4,1}^{2}}\geq 0$$
      because $j_{4,1}\geq\sqrt{24} \geq\sqrt{\frac{648}{33}}.$
   \item For $k=7$
    $$C_{7,4}(j_{4,1})=1+7^2+j_{4,1}\frac{J_{3}^{'}(j_{4,1})}{J_{3}(j_{4,1})}+j_{4,1}\frac{J_{11}^{'}(j_{4,1})}{J_{11}(j_{4,1})}.$$
     We have $j_{4,1}\leq j_{11,1}^{'}$. By (\ref{J.L1}), this implies $j_{4,1}\frac{J_{11}^{'}(j_{4,1})}{J_{11}(j_{4,1})}\geq 0$.\\
      By using (\ref{Jmm1}), one gets $\;1+7^2+j_{4,1}\frac{J_{3}^{'}(j_{4,1})}{J_{3}(j_{4,1})}= 53\geq 0$.
   \item For $k=9$
       $$C_{9,4}(j_{4,1})=1+9^2+j_{4,1}\frac{J_{5}^{'}(j_{4,1})}{J_{5}(j_{4,1})}+j_{4,1}\frac{J_{13}^{'}(j_{4,1})}{J_{13}(j_{4,1})}.$$
     We have $j_{4,1}\leq j_{13,1}^{'}$. By (\ref{J.L1}), this implies $j_{4,1}\frac{J_{13}^{'}(j_{4,1})}{J_{13}(j_{4,1})}\geq 0$. \\
     By using (\ref{Jmp1}), one gets $\;1+9^2+j_{4,1}\frac{J_{5}^{'}(j_{4,1})}{J_{5}(j_{4,1})}= 77\geq 0$.
 \end{itemize}
     \item[$\bullet$] For $j_{6,2}$:\\
     We have $j_{6,2}\leq j_{12,1}^{'}$ (see table 1 and table 2)in Appendix D.
      $$j_{6,2}\leq j_{k-6,1}^{'}\Rightarrow k-6\geq 12\Rightarrow k\geq 18$$
       So by (\ref{JL}), $\;C_{k,6}(j_{6,2})\geq 0 \mbox{ for } k\geq 18$.\\
 \begin{itemize}
   \item For $k=1$
      $$C_{1,6}(j_{6,2})=1+1^2+j_{6,2}\frac{J_{5}^{'}(j_{6,2})}{J_{5}(j_{6,2})}+j_{6,2}\frac{J_{7}^{'}(j_{6,2})}{J_{7}(j_{6,2})}.$$
By using (\ref{Jmm1}) and (\ref{Jmp1}),one gets $\;C_{1,6}(j_{6,2})=0$.

   \item For $k=3$
   $$C_{3,6}(j_{6,2})=1+3^2+j_{6,2}\frac{J_{3}^{'}(j_{6,2})}{J_{3}(j_{6,2})}+j_{6,2}\frac{J_{9}^{'}(j_{6,2})}{J_{9}(j_{6,2})}.$$
      By using (\ref{eq.AAA1}) and (\ref{Jmp3}),
      $$C_{3,6}(j_{6,2})=\frac{-2(j_{6,2}^{2}-(-472+24\sqrt{449}))(j_{6,2}^{2}+(472+24\sqrt{449}))}{(80-j_{6,2}^{2})(224-j_{6,2}^{2})}\geq 0$$

       because $\sqrt{-472+24\sqrt{449}}\leq \sqrt{80}\leq j_{6,2}\leq \sqrt{224}$.

   \item For $k=5$
   $$C_{5,6}(j_{6,2})=1+5^2+j_{6,2}\frac{J_{1}^{'}(j_{6,2})}{J_{1}(j_{6,2})}+j_{6,2}\frac{J_{11}^{'}(j_{6,2})}{J_{11}(j_{6,2})}.$$
      By using (\ref{E}) and (\ref{D}), one has\\

      $C_{5,6}(j_{6,2})=16+\frac{-j_{6,2}^{2}(480-16j_{6,2}^{2})}{(j_{6,2}^{2}-(72+8\sqrt{51}))(j_{6,2}^{2}-(72-8\sqrt{51}))}$
      $$
      \quad \quad +\frac{j_{6,2}^{2}(4032-32j_{6,2}^{2})}{(j_{6,2}^{2}-(432+48\sqrt{46}))(j_{6,2}^{2}-(432-48\sqrt{46}))}\geq 0,$$

       since $\sqrt{(432-48\sqrt{46})}\leq \sqrt{\frac{4032}{32}}\leq j_{6,2} \leq \sqrt{(432+48\sqrt{46})}$ and \\
       $j_{6,2}\geq \sqrt{72+8\sqrt{51}}\geq \sqrt{\frac{480}{16}}\geq \sqrt{72-8\sqrt{51}}.$

   \item For $k=7$
    $$C_{7,6}(j_{6,2})=1+7^2+j_{6,2}\frac{J_{1}^{'}(j_{6,2})}{J_{1}(j_{6,2})}+j_{6,2}\frac{J_{13}^{'}(j_{6,2})}{J_{13}(j_{6,2})}.$$
     We have $j_{6,2}\leq j_{13,1}^{'}$. By (\ref{J.L1}), this implies $j_{6,2}\frac{J_{13}^{'}(j_{6,2})}{J_{13}(j_{6,2})}\geq 0$. \\
     By using (\ref{D}), one has $$1+7^2+j_{6,2}\frac{J_{1}^{'}(j_{6,2})}{J_{1}(j_{6,2})}=1+7^2+1-\frac{j_{6,2}^{2}(480-16j_{6,2}^{2})}{1920-144j_{6,2}^{2}+j_{6,2}^{4}}$$
     $$
     \quad\quad=51+\frac{-j_{6,2}^{2}(480-16j_{6,2}^{2})}{(j_{6,2}^{2}-(72+8\sqrt{51}))(j_{6,2}^{2}-(72-8\sqrt{51}))}\geq 0,$$
     because $j_{6,2}\geq\sqrt{72+8\sqrt{51}} \geq\sqrt{\frac{480}{10}}\geq\sqrt{72-8\sqrt{51}}.$

   \item For $k=9$
    $$C_{9,6}(j_{6,2})=1+9^2+j_{6,2}\frac{J_{3}^{'}(j_{6,2})}{J_{3}(j_{6,2})}+j_{6,2}\frac{J_{15}^{'}(j_{6,2})}{J_{15}(j_{6,2})}.$$
     We have $j_{6,2}\leq j_{15,1}^{'}$. By (\ref{J.L1}), this implies $j_{6,2}\frac{J_{15}^{'}(j_{6,2})}{J_{15}(j_{6,2})}\geq 0$.\\
      By using (\ref{eq.AAA1}), one gets
      $$1+9^2+j_{6,2}\frac{J_{3}^{'}(j_{6,2})}{J_{3}(j_{6,2})}=\frac{6800-95j_{6,2}^{2}}{80-j_{6,2}^{2}}\geq 0,$$
     because $j_{6,2}\geq\sqrt{80} \geq\sqrt{\frac{6800}{95}}$.

   \item For $k=11$
   $$C_{11,6}(j_{6,2})=1+11^2+j_{6,2}\frac{J_{5}^{'}(j_{6,2})}{J_{5}(j_{6,2})}+j_{6,2}\frac{J_{17}^{'}(j_{6,2})}{J_{17}(j_{6,2})}.$$
     We have $j_{6,2}\leq j_{17,1}^{'}$. By (\ref{J.L1}), this implies $j_{6,2}\frac{J_{17}^{'}(j_{6,2})}{J_{17}(j_{6,2})}\geq 0$.\\
      By using (\ref{Jmm1}), $\;1+11^2+j_{6,2}\frac{J_{5}^{'}(j_{6,2})}{J_{5}(j_{6,2})}=127\geq 0.$

   \item For $k=13$
    $$C_{13,6}(j_{6,2})=1+13^2+j_{6,2}\frac{J_{7}^{'}(j_{6,2})}{J_{7}(j_{6,2})}+j_{6,2}\frac{J_{19}^{'}(j_{6,2})}{J_{19}(j_{6,2})}.$$
     We have $j_{6,2}\leq j_{19,1}^{'}$. By (\ref{J.L1}), this implies $j_{6,2}\frac{J_{19}^{'}(j_{6,2})}{J_{19}(j_{6,2})}\geq 0$. \\
     By using (\ref{Jmp1}), one has $\;1+13^2+j_{6,2}\frac{J_{7}^{'}(j_{6,2})}{J_{7}(j_{6,2})}=163\geq 0$.

   \item For $k=15$
      $$C_{15,6}(j_{6,2})=1+15^2+j_{6,2}\frac{J_{9}^{'}(j_{6,2})}{J_{9}(j_{6,2})}+j_{6,2}\frac{J_{21}^{'}(j_{6,2})}{J_{21}(j_{6,2})}.$$
     We have $j_{6,2}\leq j_{21,1}^{'}$. By (\ref{J.L1}), this implies $j_{6,2}\frac{J_{21}^{'}(j_{6,2})}{J_{21}(j_{6,2})}\geq 0$.\\
      By using (\ref{Jmp3}), one has
       $$1+15^2+j_{6,2}\frac{J_{9}^{'}(j_{6,2})}{J_{9}(j_{6,2})}=\frac{48608-203j_{6,2}^{2}}{224-j_{6,2}^{2}}\geq 0,$$
     because $j_{6,2}\leq \sqrt{224} \leq\sqrt{\frac{48608}{203}}$.

   \item For $k=17$
        $$C_{17,6}(j_{6,2})=1+17^2+j_{6,2}\frac{J_{11}^{'}(j_{6,2})}{J_{11}(j_{6,2})}+j_{6,2}\frac{J_{23}^{'}(j_{6,2})}{J_{23}(j_{6,2})}.$$
     We have $j_{6,2}\leq j_{23,1}^{'}$. By (\ref{J.L1}), this implies $j_{6,2}\frac{J_{23}^{'}(j_{6,2})}{J_{23}(j_{6,2})}\geq 0$.\\
      By using (\ref{Jmp3}), one gets \\

      $1+17^2+j_{6,2}\frac{J_{11}^{'}(j_{6,2})}{J_{11}(j_{6,2})}=279+\frac{j_{6,2}^{2}(4032-32j_{6,2}^{2})}{(j_{6,2}^{2}-(432+48\sqrt{46}))(j_{6,2}^{2}-(432-48\sqrt{46}))}\geq 0,$ \\

since $\sqrt{(432-48\sqrt{46})}\leq \sqrt{\frac{4032}{32}}\leq j_{6,2} \leq \sqrt{(432+48\sqrt{46})}$ .
 \end{itemize}
\end{description}
\end{proof}
\section{Proof of inequalities (\ref{C3mC5m})- (\ref{C38andC58})}
\label{sec:AppendixB}
By using (\ref{Jmp3}) and (\ref{eq.AAA1}), we have that
$$
C_{3,m}(j_{m,p})=4-\frac{2j_{m,p}^{2}(m-1)}{4(m-2)(m-1)-j_{m,p}^{2}}+\frac{2j_{m,p}^{2}(m+1)}{4(m+2)(m+1)-j_{m,p}^{2}}.
$$
Let us define
$$
C_{3,m}(x)=4-\frac{2x^{2}(m-1)}{4(m-2)(m-1)-x^{2}}+\frac{2x^{2}(m+1)}{4(m+2)(m+1)-x^{2}}.
$$
for $x\in I$, where $I=[\sqrt{m(m+2)},2\sqrt{(m-1)(m-2)}[\cup ]2\sqrt{(m+1)(m+2)},+\infty[$.\\

After simplification, we obtain
$$
C_{3,m}(x)=\frac{64m^{4}-320 m^{2}+256-x^{2}(64m^{2}+32)}{(4(m-2)(m-1)-x^{2})(4(m+2)(m+1)-x^{2})}
$$
\begin{description}
 \item[$\bullet$] For $\sqrt{m(m+2)}<x<2\sqrt{(m-1)(m-2)}$:\\

 We have that:
 $$
 (4(m-2)(m-1)-x^{2}) >0\;,\;(4(m+2)(m+1)-x^{2})>0
 $$
 and
 \begin{eqnarray*}
  64m^{4}-320 m^{2}+256-x^{2}(64m^{2}+32)  &<& 64m^{4}-320 m^{2}+256-m(m+2)(64m^{2}+32) \\
                                           &=& -128m^{3}-348m^{2}-64m+256<0\;\forall m\geq 4.
 \end{eqnarray*}

  \item[$\bullet$] For $2\sqrt{(m+1)(m+2)}<x$:\\

  We have that:
  $$
(4(m-2)(m-1)-x^{2}) <0\;,\;(4(m+2)(m+1)-x^{2})<0$$
 and
  \begin{eqnarray*}
  64m^{4}-320 m^{2}+256-x^{2}(64m^{2}+32)  &<& 64m^{4}-320 m^{2}+256-4(m+1)(m+2)(64m^{2}+32) \\
                                           &=& -192m^{4}-768m^{3}-960m^{2}-384<0\; \forall m\geq 4.
 \end{eqnarray*}
\end{description}

Finally, we deduce that

$$\forall m\geq 4, C_{3,m}(x)<0\;\mbox{for}\; x\in I.$$

Furthermore, from the lower bound for the first zeros of Bessel functions\\

$j_{m,1}\geq \sqrt{m(m+2)}\;\forall m\in \mathbb{N}$ (see \cite{wat1995}, p:486), we have that
\begin{equation}
C_{3,m}(j_{m,p})<0, \forall m\geq 4 \mbox{ and } j_{m,p}\in I.
\label{C3minf0}
\end{equation}

$\bullet$ Proof of  (\ref{C3mC5m}):\\
By using (\ref{E}) and (\ref{D}), we have
$$
C_{5,m}(j_{m,p})=16+\frac{\jnpp \Bigg(8(m+3)(m+2)(m+1)-4\jnpp (m+2)\Bigg)}{16(m+4)(m+3)(m+2)(m+1)-4\jnpp (m+2)(3m+9)+j_{m,p}^{4}}
$$
$$
-\frac{\jnpp \Bigg(8(m-3)(m-2)(m-1)-4\jnpp (m-2)\Bigg)}{16(m-4)(m-3)(m-2)(m-1)-4\jnpp (m-2)(3m-9)+j_{m,p}^{4}}.
$$
Let us define

$$C_{5,m}(x)=16+\frac{(x^2(8(m+3)(m+2)(m+1)-4(m+2)x^2))}{(16(m+4)(m+3)(m+2)(m+1)-4(m+2)(3m+9)x^2+x^4)}
$$
$$
\quad-\frac{(x^2(8(m-3)(m-2)(m-1)-4(m-2)x^2))}{(16(m-4)(m-3)(m-2)(m-1)-4(m-2)(3m-9)x^2+x^4)}
$$

for $x\in V$, where $V=[2\sqrt{(m-1)(m-2)},2\sqrt{(m+1)(m+2)}]$.\\

After simplification, we obtain $\displaystyle C_{5,m}(x)=\frac{-32P(x)}{Q(x)R(x)}$ where

$$
P(x)=x^{6}(18m^{2}+33)+x^{4}(378m^{2}-114m^{4}-2568)+x^{2}(224m^{6}-2176m^{4}-1504m^{2})
$$
$$
\quad \quad +(3840m^{6}-128m^{8}-34944m^{4}+104960m^{2}-73728)
$$
and

$$Q(x)=16(m+4)(m+3)(m+2)(m+1)-4(m+2)(3m+9)x^2+x^4$$

and

$$R(x)=16(m-4)(m-3)(m-2)(m-1)-4(m-2)(3m-9)x^2+x^4.$$

For a fixed $m$,

$$Q^{'}(x)=4x(x^2-6(m+2)(m+3)< 0\;\mbox{for}\;x \in V\;\mbox{and}\; m\geq 9.$$

This implies\\

$Q(x)\leq Q(2\sqrt{(m-1)(m-2)})=-16m^4-32m^3+1104m^2+992m-128 <0\;\mbox{for}\;  m\geq 9$.\\

For a fixed $m$,
$$R^{'}(x)=4x(x^2-6(m-2)(m-3)).$$

For $x \in V\;\mbox{and}\; m\geq 21$, $R'(x)<0$.

And for $9\leq m \leq 20$, one gets

\begin{itemize}
  \item $R^{'}(x)<0$, for $2\sqrt{(m-1)(m-2)}<x<\sqrt{6(m-2)(m-3)}$
  \item $R^{'}(x)>0$, for $\sqrt{6(m-2)(m-3)}<x<2\sqrt{(m-1)(m-2)}$
\end{itemize}
This implies \\
\begin{equation*}
 R(x)  \leq  \displaystyle \sup{(R(2\sqrt{(m-1)(m-2)}),R(2\sqrt{(m+1)(m+2)}))}<0 \;\mbox{ for } \;9\leq m\leq 20
\end{equation*}
Then, $$
R(x)< 0,\mbox{ for } x\in V \mbox{ and } m\geq 9.
$$
Therefore,
$$Q(x)R(x)>0 \;\mbox{for}\; x\in V\;\mbox{and}\;m\geq 9$$.

For a fixed $m$\\

$ P^{'}(x)=x^{5}(108m^2+198)+x^{3}(-456m^4+1512m^2-10272)+x(448m^6-4352m^4-3008m^2+52992)\geq 0$ for $x\in V$ and $m\geq 9$. This implies \\

$P(x)\geq P(2\sqrt{(m-2)(m-1)})=96 m^8 - 2112 m^7 + 19392m^6-79872 m^5 +117600 m^4 +17472 m^3 -162432 m^2 +99072 m-9216
 = 96 (m - 1)^2 (m - 2) (m^5 - 18 m^4 + 125 m^3 - 240 m^2 - 396 m + 48)
>0 \mbox{ for } m \geq 9$.

Finally,
$$C_{5,m}(x)=\frac{-32P(x)}{Q(x)R(x)}\leq 0;\;\mbox{for}\;x \in V\;\mbox{and}\;m\geq 9.$$
We conclude that for
$$m\geq 9,\; C_{5,m}(j_{m,p})< 0\;\mbox{for}\;j_{m,p}\in [2\sqrt{(m-1)(m-2)},2\sqrt{(m+1)(m+2)}].$$

$\bullet$ Proof of (\ref{C31})
$$
C_{3,1}(j_{1,p})=1+3^{2}+j_{1,p}\frac{J_{3-1}^{'}(j_{1,p})}{J_{3-1}(j_{1,p})}+j_{1,p} \frac{J_{3+1}^{'}(j_{1,p})}{J_{3+1}(j_{1,p})}
$$
$$
\quad\quad= 1+9+j_{1,p}\frac{J_{2}^{'}(j_{1,p})}{J_{2}(j_{1,p})}+j_{1,p}\frac{J_{4}^{'}(j_{1,p})}{J_{4}(j_{1,p})}.
$$

Furthermore, from (\ref{Jmp1}) and (\ref{Jmp3}), and for $m=1$, we have

$$
j_{1,p} \frac{J_{2}^{'}(j_{1,p})}{J_{2}(j_{1,p})}=-2
\quad\mbox{and}\quad
j_{1,p} \frac{J_{4}^{'}(j_{1,p})}{J_{4}(j_{1,p})}=-4+\frac{4j_{1,p}^{2}}{24-j_{1,p}^{2}}.
$$
Therefore,
$
\;C_{3,1}(j_{1,p})=4+\frac{4j_{1,p}^{2}}{24-j_{1,p}^{2}}
=\frac{96}{24-j_{1,p}^{2}}< 0$
for $j_{1,p}\geq \sqrt{24}=2\sqrt{6}\simeq 4.89$.
Finally we have that $C_{3,1}(j_{1,p})< 0$, $\forall p\geq 2$.
\\
$\bullet$ Proof of (\ref{C32})
$$
C_{3,2}(j_{2,p})=1+3^{2}+j_{2,p}\frac{J_{3-2}^{'}(j_{2,p})}{J_{3-2}(j_{2,p})}+j_{2,p} \frac{J_{3+2}^{'}(j_{2,p})}{J_{3+2}(j_{2,p})}
$$
$$
\quad\quad= 1+9+j_{2,p}\frac{J_{1}^{'}(j_{2,p})}{J_{1}(j_{2,p})}+j_{2,p}\frac{J_{5}^{'}(j_{2,p})}{J_{5}(j_{2,p})}.
$$
Furthermore, by using (\ref{Jmm1}) and (\ref{Jmp3}) for $m=2$, we have
$$
j_{2,p} \frac{J_{1}^{'}(j_{2,p})}{J_{1}(j_{2,p})}=1
\quad\mbox{and}\quad
j_{2,p} \frac{J_{5}^{'}(j_{2,p})}{J_{5}(j_{2,p})}=-5+\frac{6j_{2,p}^{2}}{48-j_{2,p}^{2}},
$$
so
$$
C_{3,2}(j_{2,p})=10+1-5+\frac{6j_{2,p}^{2}}{48-j_{2,p}^{2}}=6+\frac{6j_{2,p}^{2}}{48-j_{2,p}^{2}}
$$
$$
=\frac{288}{48-j_{2,p}^{2}}< 0\quad for \quad j_{2,p}\geq \sqrt{48}=4\sqrt{3}\simeq 6.92
$$
Finally we have $C_{3,2}(j_{2,p})< 0$, $\forall p\geq 2$.\\
$\bullet$ Proof of (\ref{C33}):\\
By using (\ref{eq.AAA1}) and (\ref{Jmp3})
\begin{eqnarray*}
  C_{3,3}(j_{3,p}) &=& 4+\frac{-4j_{3,p}^{2}}{8-j_{3,p}^{2}}+\frac{8j_{3,p}^{2}}{80-j_{3,p}^{2}} \\
                   &=& \frac{2560-608j_{3,p}^{2}}{(8-j_{3,p}^{2})(80-j_{3,p}^{2})}<0\quad \forall p\geq 2,
\end{eqnarray*}
because $\sqrt{\frac{2560}{608}}\leq \sqrt{8} \leq \sqrt{80} \leq j_{3,p}\; \forall p\geq 2$.

$\bullet$ Proof (\ref{C34}):\\
We have $j_{4,p}\geq 2\sqrt{(4+1)(4+2)}(\simeq 10.94)$ for $p\geq 2$. By using (\ref{C3minf0}),
$$
C_{3,4}(j_{4,p})\leq 0\quad\forall p\geq 2.
$$
$\bullet$ Proof (\ref{C35}):\\
We have $j_{5,p}\geq 2\sqrt{(5+1)(5+2)}(\simeq 12.96)$ for $p\geq 3$. By using (\ref{C3minf0}),
$$
C_{3,5}(j_{5,p})\leq 0 \quad \forall p\geq 3.
$$
$\bullet$ Proof (\ref{C36}):\\
We have $j_{6,p}\geq 2\sqrt{(6+1)(6+2)}(\simeq 15.87)$ for $p\geq 3$. By using (\ref{C3minf0}),
$$
C_{3,6}(j_{6,p})\leq 0\quad \forall p\geq 3.
$$
$\bullet$ Proof(\ref{C37}):\\
We have $j_{7,p}\geq 2\sqrt{(7+1)(7+2)}(\simeq 16.97)$ for $p\geq 3$. By using (\ref{C3minf0}),
$$
C_{3,7}(j_{7,p})\leq 0 \quad \forall p\geq 3.
$$
$\bullet$ Proof(\ref{C38}):\\
We have $j_{8,p}\geq 2\sqrt{(8+1)(8+2)}(\simeq 18.97)$ for $p\geq 3$. By using (\ref{C3minf0}),
$$
C_{3,8}(j_{8,p})\leq 0 \quad \forall p\geq 3.
$$
$\bullet$ The first inequality of (\ref{C38andC58}):\\
We have $(\mbox{for}\;m=8,\;8.94\simeq\sqrt{8(8+2)}\leq j_{8,1}\leq 2\sqrt{(8-1)(8-2)}(\simeq 12.96)$. By(\ref{C3minf0}), we deduce
$$
C_{3,8}(j_{8,1})\leq 0.
$$
$\bullet$ Proof (\ref{App41}):
$$
C_{5,3}(j_{3,p})=1+5^{2}+j_{3,p}\frac{J_{5-3}^{'}(j_{3,p})}{J_{5-3}(j_{3,p})}+j_{3,p} \frac{J_{5+3}^{'}(j_{3,p})}{J_{5+3}(j_{3,p})}
$$
$$
\quad\quad= 1+25+j_{3,p}\frac{J_{2}^{'}(j_{3,p})}{J_{2}(j_{3,p})}+j_{3,p}\frac{J_{8}^{'}(j_{3,p})}{J_{8}(j_{3,p})}.
$$
From (\ref{Jmm1}) and for $m=3$, we have that $j_{3,p} \frac{J_{2}^{'}(j_{3,p})}{J_{2}(j_{3,p})}=2$. By (\ref{E}) for $m=3$, we have $
j_{3,p} \frac{J_{8}^{'}(j_{3,p})}{J_{8}(j_{3,p})}=-8+\frac{j_{3,p}^{2}(960-20j_{3,p}^{2})}{13440-360j_{3,p}^{2}+j_{3,p}^{4}}
$. Therefore,
$$
C_{5,3}(j_{3,p})=20+\frac{j_{3,p}^{2}(960-20j_{3,p}^{2})}{13440-360j_{3,p}^{2}+j_{3,p}^{4}}=\frac{268800-6240j_{3,p}^{2}}{13440-360j_{3,p}^{2}+j_{3,p}^{4}}
$$
$$
\quad=\frac{268800-6240j_{3,p}^{2}}{(j_{3,p}^{2}-(180-\sqrt{18960}))(j_{3,p}^{2}-(180+\sqrt{18960}))}.
$$
Since $j_{3,5}\simeq 19.4094\geq \sqrt{(180+\sqrt{18960})}\simeq 17.82\geq \sqrt{\frac{268800}{6240}}\simeq 6.56$, then $268800-6240j_{3,p}^{2}<0$ and $(j_{3,p}^{2}-(180-\sqrt{18960}))(j_{3,p}^{2}-(180+\sqrt{18960}))>0$\;$\forall p\geq 5$.\\

Then, $$C_{5,3}(j_{3,p})<0\;\mbox{for}\;\forall p\geq 5.$$

$\bullet$ Proof (\ref{C56})
$$
C_{5,6}(j_{6,1})=1+5^{2}+j_{6,1}\frac{J_{1}^{'}(j_{6,1})}{J_{1}(j_{6,1})}+j_{6,1} \frac{J_{11}^{'}(j_{6,1})}{J_{11}(j_{6,1})}.
$$

By using (\ref{E}) and (\ref{D}) for $m=6$ and $p=1$, we have that
\begin{eqnarray*}
 C_{5,6}(j_{6,1})  &=& 16+\frac{j_{6,1}^{2}(4032-32j_{6,1}^{2})}{80640-864j_{6,1}^{2}+j_{6,1}^{4}}-\frac{j_{6,1}^{2}(480-16j_{6,1}^{2})}{1920-144j_{6,1}^{2}+j_{6,1}^{4}} \\
                   &=& \frac{j_{6,1}^{2}(4032-32j_{6,1}^{2})}{80640-864j_{6,1}^{2}+j_{6,1}^{4}}-20 +36 -\frac{j_{6,1}^{2}(480-16j_{6,1}^{2})}{1920-144j_{6,1}^{2}+j_{6,1}^{4}}\\
                   &=& \frac{-52(j_{6,1}^{2}-(\frac{2664}{13}-\frac{24}{13}\sqrt{3221}))(j_{6,1}^{2}-(\frac{2664}{13}+\frac{24}{13}\sqrt{3221}))}{(j_{6,1}^{2}-(432+48\sqrt{46}))(j_{6,1}^{2}-(432-48\sqrt{46}))} \\
                   &+& \quad \frac{52(j_{6,1}^{2}-(\frac{708}{13}+\frac{12}{13}\sqrt{1921}))(j_{6,1}^{2}-(\frac{708}{13}-\frac{12}{13}\sqrt{1921}))}{(j_{6,1}^{2}-(72+8\sqrt{51}))(j_{6,1}^{2}-(72-8\sqrt{51}))}\leq 0,
\end{eqnarray*}
as $\sqrt{\frac{708}{13}-\frac{12}{13}\sqrt{1921}}\leq \sqrt{72-8\sqrt{51}}\leq \sqrt{\frac{708}{13}+\frac{12}{13}\sqrt{1921}}\leq j_{6,1}\leq \sqrt{\frac{2664}{13}-\frac{24}{13}\sqrt{3221}}\leq \sqrt{432-48\sqrt{46}}\leq \sqrt{72+8\sqrt{51}}\leq \sqrt{\frac{2664}{13}+\frac{24}{13}\sqrt{3221}}\leq \sqrt{432-48\sqrt{46}}$.\\

$\bullet$ Proof  for the second inequality of (\ref{C38andC58}):
$$
C_{5,8}(j_{8,2})=1+5^{2}+j_{8,2}\frac{J_{3}^{'}(j_{8,2})}{J_{3}(j_{8,2})}+j_{8,2} \frac{J_{13}^{'}(j_{8,2})}{J_{13}(j_{8,2})}
$$
By using (\ref{E}) and (\ref{D}) for $m=8$ and $p=2$, we have that
\begin{eqnarray*}
 C_{5,8}(j_{8,2})  &=& 16+\frac{j_{8,2}^{2}(7920-40j_{8,2}^{2})}{190080-1320j_{8,2}^{2}+j_{8,2}^{4}}-\frac{j_{8,2}^{2}(1680-24j_{8,2}^{2})}{13440-360j_{8,2}^{2}+j_{8,2}^{4}} \\
                   &=& \frac{j_{8,2}^{2}(7920-40j_{8,2}^{2})}{190080-1320j_{8,2}^{2}+j_{8,2}^{4}}-11+27-\frac{j_{8,2}^{2}(1680-24j_{8,2}^{2})}{13440-360j_{8,2}^{2}+j_{8,2}^{4}} \\
                   &=& \frac{-51(j_{8,2}^{2}-(220+\frac{44}{17}\sqrt{1105}))(j_{8,2}^{2}-(220-\frac{44}{17}\sqrt{1105}))}{(j_{8,2}^{2}-(660+12\sqrt{1705}))(j_{8,2}^{2}-(660-12\sqrt{1705}))} \\
                   &+& \quad \frac{54(j_{8,2}^{2}-(\frac{950}{9}+\frac{2}{9}\sqrt{89545}))(j_{8,2}^{2}-(\frac{950}{9}-\frac{2}{9}\sqrt{89545}))}{(j_{8,2}^{2}-(180+4\sqrt{1185}))(j_{8,2}^{2}-(180-4\sqrt{1185}))}\leq 0,
\end{eqnarray*}

because \\

$\sqrt{\frac{950}{9}-\frac{2}{9}\sqrt{89545}}\leq \sqrt{180-4\sqrt{1185}}\leq \sqrt{220-\frac{44}{17}\sqrt{1105}} \leq \sqrt{660-12\sqrt{1705}}\leq \sqrt{\frac{950}{9}+\frac{2}{9}\sqrt{89545}}\leq j_{8,2}\leq\sqrt{220+\frac{44}{17}\sqrt{1105}}\leq \sqrt{180+4\sqrt{1185}}\leq \sqrt{660+12\sqrt{1705}}$.

\section{Roots of the Bessel function}

\begin{table}[!h]
\small
\centering
\label{Table1}
\begin{tabular}{|l|l|l|l|l|l|l|l|l|l|}
\hline
$n\setminus p$ & 1     & 2     & 3     & 4     & 5     & 6     & 7     & 8     & 9     \\ \hline
0   & 2.4048  & 5.5201  & 8.6537  & 11.7915 & 14.9309 & 18.0711 & 21.2116 & 24.3525 & 27.4935 \\ \hline
1   & 3.8317  & 7.0156  & 10.1735 & 13.3237 & 16.4706 & 19.6159 & 22.7601 & 25.9037 & 29.0468 \\ \hline
2   & 5.1356  & 8.4172  & 11.6198 & 14.7960 & 21.1170 & 27.4206 & 30.5692 & 33.7165 & 40.0084 \\ \hline
3   & 6.3802  & 9.7610  & 13.0152 & 16.2235 & 19.4094 & 22.5827 & 25.7482 & 28.9084 & 32.0649 \\ \hline
4   & 7.5883  & 11.0647 & 14.3725 & 17.6160 & 20.8269 & 24.0190 & 27.1991 & 30.3710 & 33.5371 \\ \hline
5   & 8.7715  & 12.3386 & 15.7002 & 18.9801 & 22.2178 & 25.4303 & 28.6266 & 31.8117 & 34.9888 \\ \hline
6   & 9.9361  & 13.5893 & 17.0038 & 20.3208 & 23.5861 & 26.8202 & 30.0337 & 33.2330 & 36.4220 \\ \hline
7   & 11.0864 & 14.8213 & 18.2876 & 21.6415 & 24.9349 & 28.1912 & 31.4228 & 34.6371 & 37.8387 \\ \hline
8   & 12.2251 & 16.0378 & 19.5545 & 22.9452 & 26.2668 & 29.5457 & 32.7958 & 36.0256 & 39.2404 \\ \hline
\end{tabular}
\caption{Roots $\displaystyle j_{n,p}$ of the Bessel function}
\end{table}

\begin{table}[!h]
\small
\centering
\label{Table2}
\begin{tabular}{|l|l|l|l|l|l|l|l|l|}
\hline

$n\setminus p$	&1	      &2	      &3	      &4	       &5	      &6	       &7      	&8 \\ \hline
 0      & 0         & 3.8317    & 7.0156    & 10.1735   & 13.3237   & 16.4706   & 19.6159   & 22.7601 \\ \hline			
 1     	& 1.8411  	& 5.3314  	& 8.5363 	& 11.7060  	& 14.8635  	& 18.0155  	& 21.1643 	& 24.3113 \\ \hline
 2     	& 3.0542  	& 6.7061  	& 9.9694	& 13.1703 	& 16.3475 	& 19.5129 	& 22.6715 	& 25.8260 \\ \hline
 3     	& 4.2011 	& 8.0152	& 11.3459	& 14.5858 	& 17.7887 	& 20.9724 	& 24.1448 	& 27.3100 \\ \hline
 4     	& 5.3175 	& 9.2823 	& 12.6819 	& 15.9641 	& 19.1960 	& 22.4010 	& 21.6415 	& 28.7678 \\ \hline
 5     	& 6.4156 	& 10.5198 	& 13.9871	& 17.3128 	& 20.5755 	& 23.8035 	& 25.5897 	& 30.2028 \\ \hline
 6     	& 7.5012 	& 11.7349 	& 15.2681 	& 18.6374 	& 21.9317 	& 25.1839 	& 27.0103 	& 31.6178 \\ \hline
 7     	& 8.5778 	& 12.9323 	& 16.5293 	& 19.9418 	& 23.2680 	& 26.5450 	& 29.7907 	& 33.0151 \\ \hline
 8     	& 9.6474 	& 14.1155	& 17.7740 	& 21.2290 	& 24.5871 	& 27.8892 	& 31.1553 	& 34.3966 \\ \hline
 9     	& 10.7114	& 15.2867 	& 19.0045 	& 22.5013 	& 25.8912 	& 29.2185 	& 32.5052 	& 35.7637 \\ \hline
 10     & 11.7709   & 16.4479   & 20.2230   & 23.7607   & 27.1820   & 30.5345   & 33.8420   & 37.1180 \\ \hline					
 11     & 12.8265   & 17.0603   & 21.4309   & 25.0085   & 28.4609   & 31.8384   & 35.1667   & 38.4604 \\ \hline					
 12     & 13.8788   & 18.7451   & 22.6293   & 26.2460   & 29.7290   & 33.1314   & 36.4805   & 39.7919 \\ \hline					
 13     & 14.9284   & 19.8832   & 23.8194   & 27.4743   & 30.9874   & 34.4145   & 37.7844   & 41.1135\\ \hline				
\end{tabular}
\caption{Roots $j_{n,p}^{'}$ of the derivative of the Bessel function }
\end{table}

\begin{table}[!h]
\small
\centering
\label{Table3}
\begin{tabular}{|l||l||l||l||l|}
\hline

$\lambda_{1}=j_{0,1}^{2}$&$
\lambda_{2}=\lambda_{3}=j_{1,1}^{2}$&$
\lambda_{4}=\lambda_{5}=j_{2,1}^{2}$&$
\lambda_{6}=j_{0,2}^{2}$&
$\lambda_{7}=\lambda_{8}=j_{3,1}^{2}$\\ \hline \hline $
\lambda_{9}=\lambda_{10}=j_{1,2}^{2}$&$
\lambda_{11}=\lambda_{12}=j_{4,1}^{2}$&$
\lambda_{13}=\lambda_{14}=j_{2,2}^{2}$&
$\lambda_{15}=j_{0,3}^{2}$&$
\lambda_{16}=\lambda_{17}=j_{5,1}^{2}$\\ \hline \hline $
\lambda_{18}=\lambda_{19}=j_{3,2}^{2}$&$
\lambda_{20}=\lambda_{21}=j_{6,1}^{2}$&$
\lambda_{22}=\lambda_{23}=j_{1,3}^{2}$&$
\lambda_{24}=\lambda_{25}=j_{4,2}^{2}$&$
\lambda_{26}=\lambda_{27}=j_{7,1}^{2}$\\ \hline \hline$
\lambda_{28}=\lambda_{29}=j_{2,3}^{2}$&
$\lambda_{30}=j_{0,4}^{2}$&$
\lambda_{31}=\lambda_{32}=j_{8,1}^{2}$&$
\lambda_{33}=\lambda_{34}=j_{5,2}^{2}$&$
\lambda_{35}=\lambda_{36}=j_{3,3}^{2}$ \\ \hline \hline$
\lambda_{37}=\lambda_{38}=j_{1,4}^{2}$&$
\lambda_{39}=\lambda_{40}=j_{9,1}^{2}$&$
\lambda_{41}=\lambda_{42}=j_{6,2}^{2}$&$
\lambda_{43}=\lambda_{44}=j_{4,3}^{2}$ &
$\lambda_{45}=\lambda_{46}=j_{10,1}^{2}$\\ \hline \hline $
\lambda_{47}=\lambda_{48}=j_{2,4}^{2}$&$
\lambda_{49}=\lambda_{50}=j_{7,2}^{2}$&$
\lambda_{51}=j_{0,5}^{2}$&$
\lambda_{52}=\lambda_{53}=j_{11,1}^{2}$&$
\lambda_{54}=\lambda_{55}=j_{5,3}^{2}$\\ \hline \hline$
\lambda_{56}=\lambda_{57}=j_{8,2}^{2}$&$
\lambda_{58}=\lambda_{59}=j_{3,4}^{2}$&
$\lambda_{60}=\lambda_{61}=j_{1,5}^{2}$&$
\lambda_{62}=\lambda_{63}=j_{1,5}^{2}$&$
\lambda_{64}=\lambda_{65}=j_{12,1}^{2}$\\ \hline \hline$
\lambda_{66}=\lambda_{67}=j_{6,3}^{2}$& $
\lambda_{68}=\lambda_{69}=j_{9,2}^{2}$&$
\lambda_{70}=\lambda_{71}=j_{4,4}^{2}$&$
\lambda_{72}=\lambda_{73}=j_{13,1}^{2}$&$
\lambda_{74}=\lambda_{75}=j_{2,5}^{2}$\\ \hline \hline
$\lambda_{76}=j_{0,6}^{2}$&$
\lambda_{77}=\lambda_{78}=j_{7,3}^{2}$&$
\lambda_{79}=\lambda_{80}=j_{10,2}^{2}$&$
\lambda_{81}=\lambda_{82}=j_{14,1}^{2}$&$
\lambda_{83}=\lambda_{84}=j_{5,4}^{2}$\\ \hline \hline$
\lambda_{85}=\lambda_{86}=j_{3,5}^{2}$&$
\lambda_{87}=\lambda_{88}=j_{8,3}^{2}$&$
\lambda_{89}=\lambda_{90}=j_{1,6}^{2}$&
$\lambda_{91}=\lambda_{92}=j_{11,2}^{2}$&$
\lambda_{93}=\lambda_{94}=j_{15,1}^{2}$\\ \hline \hline$
\lambda_{95}=\lambda_{96}=j_{6,4}^{2}$&$
\lambda_{97}=\lambda_{98}=j_{12,2}^{2}$&$
\lambda_{99}=\lambda_{100}=j_{9,3}^{2}$&$
\lambda_{101}=\lambda_{102}=j_{4,5}^{2}$&$
\lambda_{103}=\lambda_{104}=j_{16,1}^{2}$\\ \hline
\end{tabular}

\caption{links between the first 103 eigenvalues of the Dirichlet Laplacian in a disk of radius $1$ and zeros of Bessel functions $j_{m,p}$}
\end{table}

\newpage
\section*{Acknowledgments}

The authors would like to thank Antoine Henrot for having proposed the subject of this paper in the Master thesis when the first author was in Institute Elie Cartan of Nancy. The authors thank Gisella Croce for some useful discussions.


\end{document}